\documentclass[11pt]{amsart}
\usepackage{amsmath,amsfonts,amssymb,mathrsfs}
\usepackage{amssymb,mathrsfs,graphicx,extpfeil}
\usepackage{amsmath,amsfonts,amssymb,amscd,amsthm,bbm}

\usepackage{epsfig}
\usepackage{indentfirst, latexsym, amssymb, enumerate,amsmath,graphicx}
\usepackage{float}
\usepackage{epsfig}

\usepackage{caption}
\usepackage{subcaption}
\usepackage{extpfeil}
\usepackage{graphicx,colortbl}
\usepackage{epsfig}
\usepackage{caption}
\usepackage{bm}
\usepackage{tikz}
\usetikzlibrary{matrix,shapes,arrows,positioning,chains}

\usepackage{amsmath}
\title{On a generalized Kuramoto model with relativistic effects and emergent dynamics}

\author[C. Min]{Chan Ho Min}
\address[C. Min]{\newline Department of Financial engineering \newline Ajou University, Suwon 16499, Republic of Korea}
\email{chanhomin@ajou.ac.kr}

\author[Ahn]{Hyunjin Ahn}
\address[Hyunjin Ahn]{\newline Department of Mathematical Sciences\newline Seoul National University, Seoul 08826, Republic of Korea}
\email{yagamelaito@snu.ac.kr}

\author[S-Y. Ha]{Seung-Yeal Ha}
\address[S.-Y. Ha]{\newline Department of Mathematical Sciences and Research Institute of Mathematics \newline Seoul National University, Seoul 08826 and \newline
Korea Institute for Advanced Study, Hoegiro 85, Seoul, 02455, Republic of 
Korea} \email{syha@snu.ac.kr}

\author[M. Kang]{Myeongju Kang}
\address[M. Kang]{\newline Department of Mathematical Sciences and Research Institute of Mathematics \newline Seoul National University, Seoul 08826, Republic of Korea and}
\email{bear0117@snu.ac.kr}

\newtheorem{theorem}{Theorem}[section]
\newtheorem{lemma}{Lemma}[section]

\newtheorem{proposition}{Proposition}[section]

\newtheorem{remark}{Remark}[section]

\newtheorem{definition}{Definition}[section]

\newcommand{\bbr}{\mathbb R}

\newcommand{\bbs}{\mathbb S}

\newcommand{\bbt} {\mathbb T}

\newcommand{\bx}{\mbox{\boldmath $x$}}

\newcommand{\bv}{\mbox{\boldmath $v$}}

\newcommand{\cD}{\mathcal D}

\def\charf {\mbox{{\text 1}\kern-.30em {\text l}}}
\calclayout

\begin{document}
%%%%%%%%%%%%%%%%
%%%%%%%%%%%%%%%%
\tikzstyle{block} = [rectangle, draw, %fill=blue!20, 
    text width=15em, text centered, rounded corners, minimum height=3em]
\tikzstyle{line} = [draw, -latex']

\date{\today}


\thanks{Acknowledgment: The work of C. Min was supported by the new faculty research fund of Ajou University, the work of S.-Y. Ha is partially supported by National Research Foundation of Korea Grant NRF-2020R1A2C3A01003881), and the work of M. Kang was supported by the National Research Foundation of Korea(NRF) grant funded by the Korea government(MSIP)(2016K2A9A2A13003815)}

\begin{abstract}
We propose a generalized Kuramoto model with relativistic effects and investigate emergent asymptotic behaviors. The proposed generalized Kuramoto model incorporates relativistic Kuramoto(RK) type models which can be derived from the relativistic Cucker-Smale (RCS) on the unit sphere under suitable approximations. We present several sufficient frameworks leading to complete synchronization in terms of initial data and system parameters. For the relativistic Kuramoto model,  we show that it can be reduced to the Kuramoto model in any finite time interval in a non-relativistic limit. We also provide several numerical examples for two approximations of the relativistic Kuramoto model, and compare them with analytical results. 
\end{abstract}

\maketitle

\centerline{\date}

%\tableofcontents

%%%%%%%%%%%%%%%%%%%%%%%%%%%%%%%%%%%%%%%%%%%%%%%%%%%%%%%%%%%%%%%%%%%%%%%%%%%%%%%%%%%%%%%%%%%%%%%%%%%%%%%%%%%%%%%%%%%%%%%%%I

\section{Introduction} \label{sec:1}
\setcounter{equation}{0}
Collective phenomena are ubiquitous in nature and human societies, e.g., aggregation of bacteria \cite{T-B}, flocking of birds \cite{C-Sm,H-J-K,T-T}, synchronization of pacemaker cells and fireflies \cite{B-C-M1,B-C-M,B-D-P,B-B,C-S,Ku2,Pe,P-R}, swarming of fish \cite{D-M1,D-M2,D-M3}, etc. For a brief introduction on the subject, we refer to review papers and books  \cite{A-B,P-R,St,VZ,Wi1}. To begin with,  we consider the nonrelativistic Kuramoto model. Let $\theta_i = \theta_i(t)$ be the phase of the $i$-th Kuramoto oscillator whose dynamics is governed by the system of first-order ordinary differential equations:
\begin{equation} \label{A-1}
{\dot \theta}_i = \nu_i + \frac{\kappa}{N} \sum_{j= 1}^{N} \sin(\theta_j - \theta_i), \quad \forall~i = 1, \cdots, N,
\end{equation}
where $\nu_i$ and $\kappa$ are natural frequency of the $i$-th oscillator and nonnegative coupling strength, respectively. The emergent dynamics of \eqref{A-1} has been extensively studied in literature from various viewpoints, to name a few, complete synchronization \cite{B-D-P,C-H-J-K,C-S,D-X,H-K-R}, critical coupling strength \cite{D-B},  uniform mean field limit \cite{H-K-P-Z,La}, gradient flow formulation \cite{V-W}, thermodynamic Kuramoto model \cite{H-P-R-S}, kinetic Kuramoto model \cite{B-C-M1,B-C-M,C-C-H-K-K} etc.

Next, we present a generalized phase-coupled model generalizing the Kuramoto model \eqref{A-1}. Let $F$ be an odd and continuously differentiable monotone increasing in the interval $(-L, L)$:
\begin{equation} \label{A-2}
 F^{\prime}(\omega) > 0, \quad F(-\omega) = -F(\omega), \quad \forall~\omega \in (-L, L). 
 \end{equation}
Then, the generalized Kuramoto model to be discussed in this paper reads as follows.
\begin{equation} \label{A-3}
F \big( \dot\theta_i \big) = \nu_i +\frac{\kappa}{N} \sum_{j=1}^N \sin(\theta_j -\theta_i), \quad \forall~i=1,\cdots, N.
\end{equation}
Of course, for the well-definedness of $\dot\theta_i$, the R.H.S. of \eqref{A-3} must be in the range of $F$. Note that the choice $F(\omega) = \omega$ with $L = \infty$ satisfies the relations \eqref{A-2} and  system \eqref{A-3} reduces to the Kuramoto model \eqref{A-1}, whereas the choice
\[   \Gamma(\omega):= \frac{1}{\sqrt{1 - \frac{|\omega|^2}{c^2}}}, \quad L = c, \quad F(\omega) = \omega \Gamma(\omega) \bigg(1+\frac{\Gamma(\omega)}{c^2}\bigg) \]
appears in the modeling of \cite{AHKS} which can be derived from the relativistic Cucker-Smale model \cite{H-K-Rug-1} (see Section \ref{sec:3}) and it satisfy the relations \eqref{A-2}. Here $c$ is the speed of light. The non-local modeling for $F({\dot \theta}_i)$ also appears in the fractional Kuramoto model \cite{H-J} which is out of our scope. In this paper, we are interested in the following two questions.\begin{itemize}
\item
(Q1):~Under what conditions on system parameters and initial data, does the generalized Kuramoto model \eqref{A-3} exhibit asymptotic synchronization?
\vspace{0.2cm}
\item
(Q2):~For the relativistic Kuramoto model
\begin{align} \label{A-4}
\begin{aligned}
& \dot\theta_i \Gamma_i \bigg(1+\frac{\Gamma_i}{c^2}\bigg) = \nu_i +\frac{\kappa}{N} \sum_{j=1}^N \sin(\theta_j -\theta_i), \quad \forall ~t > 0, \quad \forall~i=1,\cdots, N,
\end{aligned}
\end{align}
does system \eqref{A-4} converge to the Kuramoto model \eqref{A-1} as $c \to \infty$?
\end{itemize}

The purpose of this paper tries to answer the above two questions. More precisely, our main results can be summarized as follows. First, we provide sufficient frameworks leading to the complete synchronization (see Definition \ref{D2.1}) in relation with (Q1).  For an identical ensemble with the same natural frequency, phase dynamics is governed by the following system:
\begin{equation*} \label{A-4-1}
\begin{cases}
\displaystyle F\big( \dot\theta_i \big) = \nu +\frac{\kappa}{N} \sum_{j=1}^N \sin(\theta_j -\theta_i), \quad \forall~t>0, \\
\displaystyle \theta_i(0) = \theta_i^{in}, \quad \forall~i=1,\cdots, N.
\end{cases}
\end{equation*}
Our first set of main result deals with complete synchronization. Consider a homogeneous ensemble whose dynamics is governed by the following system:
 \begin{equation} \label{A-5}
F \big( \dot\theta_i \big) = \nu +\frac{\kappa}{N} \sum_{j=1}^N \sin(\theta_j -\theta_i), \quad \forall~i=1,\cdots, N.
\end{equation}
If the initial data and coupling strength satisfy
\[ {\mathcal D}(\Theta^{in}) < \pi \quad \mbox{and} \quad \kappa > 0, \]
one has an exponential synchronization: there exists a positive constant $\Lambda = \Lambda(\Theta^{in}, \nu, \kappa, (F^{-1})')$ such that  
\[ {\mathcal D}(\Theta(t)) \leq e^{-\Lambda t} {\mathcal D}(\Theta^{in}), \quad \forall~t > 0. \]
For details, see Theorem \ref{T4.1}. As another setting, we assume that initial data and coupling strength satisfy
\begin{align*}
R(0) = \bigg| \frac{1}{N} \sum_{i=1}^N e^{\mathrm i\theta_i^{in}} \bigg| > 0 \quad \mbox{and} \quad \kappa > 0,
\end{align*}
Then, the phase configuration tends to either completely synchronized state or bi-cluster state:
\begin{align*}
\lim_{t\to\infty} (\theta_i(t) -\theta_j(t)) = 0 \mod ~\pi, \quad \forall~i, j=1,\cdots, N.
\end{align*}
For details, see Theorem \ref{T4.2}. On the other hand, for a heterogeneous ensemble, if initial configuration, natural frequency, and coupling strength satisfy
\begin{align*}
& \kappa > {\mathcal D}(\Omega)>0, \quad{\mathcal  D}(\Theta^{in}) < \pi-\theta_*, \quad \theta_* :=\sin^{-1} \bigg( \frac{{\mathcal D}(\nu)}{\kappa} \bigg) \in\bigg( 0, \frac{\pi}{2} \bigg).
\end{align*}
Then, asymptotic complete synchronization emerges (see Theorem \ref{T4.3}):
\begin{align*}
\lim_{t\to\infty} \big| \dot\theta_i(t) -\dot\theta_j(t) \big| = 0.
\end{align*}
Second, our last result is concerned with the non-relativistic limit of \eqref{A-4}. More precisely, let $\Theta^c = (\theta_1^c, \cdots, \theta_N^c)$ and $\Theta^{\infty} = (\theta_1^{\infty}, \cdots, \theta_N^{\infty})$ be solutions to \eqref{A-4} and \eqref{A-1} with the same initial data, respectively. We set 
\[ \| \Theta^c - \Theta^{\infty} \|_1 := \sum_{i=1}^{N} |\theta^c_i - \theta^{\infty}_i|. \]
Then, one can derive a non-relativistic limit (see Theorem \ref{T5.1}):
\[ \lim_{c\to\infty}\sup_{0\leq t\leq T} \|\Theta^c(t) - \Theta^{\infty}(t) \|_1 = 0. \]	

\vspace{.2cm}

The rest of paper is organized as follows. In Section \ref{sec:2}, we briefly review the (non-relativistic) Kuramoto model and its emergent dynamics. In Section \ref{sec:3}, we introduce the relativistic Kuramoto model, basic structural properties and a gradient flow formulation and related convergence results.  In Section \ref{sec:4}, we study emergent properties of the relativistic Kuramoto model for homogenous and heterogeneous ensembles. In Section \ref{sec:5}, we study a non-relativistic limit from the relativistic model to the Kuramoto model, as the speed of light tends to infinity, and we  present several numerical simulations for the non-relativistic and relativistic Kuramoto models and compare them with analytical results. Finally, Section \ref{sec:6} is devoted to a brief summary of our main results and some remaining issues to be explored in a future work.  \newline

\noindent {\bf Gallery of notation}: Throughout the paper, we will use the following simplified notation:
\begin{align*}
& {\mathcal N} := \{1, \cdots, N \}, \quad \Theta := (\theta_1, \cdots, \theta_N), \quad \dot{\Theta} := ({\dot \theta}_1, \cdots, {\dot \theta}_N), \quad 
\Omega := (\nu_1, \cdots, \nu_N).
\end{align*}
We set phase and natural frequency diameters as 
\[ {\mathcal D}(\Theta) := \max_{1\leq i, j\leq N} |\theta_i -\theta_j|, \quad {\mathcal D}(\Omega) := \max_{1\leq i, j \leq N} |\nu_i-\nu_j|. \]

\section{Preliminaries} \label{sec:2}
\setcounter{equation}{0}
In this section, we first recall the nonrelativistic Kuramoto model and review the state-of-the-art on the emergent dynamics, and then we introduce the relativistic Kuramoto model and study its basic properties. 

\subsection{The Kuramoto model} \label{sec:2.1}
Let $z_i =  e^{{\mathrm i} \theta_i} \in \bbs^1$ be the position of the $i$-th rotor, and let $\theta_i$ and ${\dot \theta}_i$ denote the phase and frequency of the $i$-th oscillator, respectively. Then, the phase dynamics is governed by the following Cauchy problem: 
\begin{equation} \label{Ku}
\begin{cases}
\displaystyle \dot{\theta}_{i} = \nu_i +
\frac{\kappa}{N}\sum_{j=1}^{N}\sin(\theta_{j} - \theta_{i}), \quad \forall~t > 0, \\
\theta_i(0) = \theta_{i}^{in},  \quad \quad \forall~i =1,\cdots, N.
\end{cases} 
\end{equation}
In the sequel, we introduce a minimum material to be used crucially. First, we recall several concepts in relation with collective dynamics. 
\begin{definition}  \label{D2.1}
\emph{\cite{H-K-R}}  Let $\Theta = (\theta_1, \cdots, \theta_N)$ be a Kuramoto phase vector whose dynamics is governed by \eqref{Ku}.
\begin{enumerate}
\item
$\Theta = \Theta(t)$ is a phase-locked state of \eqref{Ku}, if all relative phase differences are constant:
\[  \theta_i(t) - \theta_j(t) = \theta_i^{in} - \theta_j^{in},  ~~\forall~t \geq 0, ~~\forall~i, j=1,\cdots, N. \] 
\item
$\Theta = \Theta(t)$ exhibits (asymptotic) complete phase-locking, if the relative phase differences converge as $t \to \infty$:
\[ \exists \lim_{t \to \infty} (\theta_i(t) - \theta_j(t)), \quad \forall~i, j=1,\cdots, N. \]
\item
$\Theta = \Theta(t)$ exhibits complete synchronization, if the relative frequency differences converge to zero as $t \to \infty$:
\[  \lim_{t \to \infty} |{\dot \theta}_i(t) - {\dot \theta}_j(t)| = 0, \quad \forall~i, j=1,\cdots, N. \]
\end{enumerate}
\end{definition}
In what follows, we briefly review an order parameter and a gradient flow formulation of $\eqref{Ku}_1$. 

\subsubsection{Order parameter}  \label{sec:2.1.1}
Let $\Theta  = \Theta(t)$ be an $N$-phase vector whose time evolution is governed by \eqref{Ku}. Then, we define real order parameters $R(\Theta)$ and $\phi(\Theta)$ by the following relation:
\begin{equation*} \label{Order}
R(\Theta) e^{\mathrm{i} \phi(\Theta)} :=\frac{1}{N}\sum_{j=1}^{N}e^{\mathrm{i}\theta_j}.
\end{equation*}
This implicit relation yields
\begin{align}
\begin{aligned} \label{R-order}
R^2 &=  R e^{\mathrm{i} \phi(\Theta)}   \overline{R e^{\mathrm{i} \phi(\Theta)} } = \frac{1}{N^2} \sum_{j, k=1}^{N}e^{\mathrm{i} (\theta_j - \theta_k)}  \\
&=  \frac{1}{N^2}  \Big( \sum_{j, k=1}^{N} \cos(\theta_j  - \theta_k) + {\mathrm i}  \sum_{j, k=1}^{N} \sin(\theta_j  - \theta_k) \Big) =  \frac{1}{N^2}  \sum_{j, k=1}^{N} \cos(\theta_j  - \theta_k).
\end{aligned}
\end{align}
Note that the amplitude order parameter $R(\Theta) \in [0, 1]$ is well-defined for all $t\ge 0$, and it is invariant under uniform rotation. It measures overall ``phase coherence" of the ensemble $\Theta$. For example, $R(\Theta)=1$ corresponds to the state in which all phases are the same, i.e., complete phase synchronization:
\[ R(\Theta) = 1 \iff \Theta = (\alpha, \cdots, \alpha) \mod 2\pi, \quad \mbox{for some } ~\alpha \in \bbr, \]
whereas $R(\Theta)  = 0$ corresponds to an incoherent state in which oscillators behave independently. On the other hand, $\phi(\Theta)$ is well defined modulo $2\pi$ if $R(\Theta)>0$, but it is meaningless when $R(\Theta)=0$. If we suppose $R(\Theta(t))>0$ for all $t$ in some time interval $\mathcal{I}$, then it is possible to choose a branch of $\phi(\Theta(t))$ smoothly on $\mathcal{I}$. As long as there is no confusion, we sometimes suppress $\Theta$-dependence on $R$ and $\phi$:
\[ R(t) := R(\Theta(t)), \quad \phi(t) := \phi(\Theta(t)), \quad \forall~t \in \mathcal{I}. \]

\subsubsection{A gradient flow formulation} \label{sec:2.1.2}
Next, we present another alternative formulation of the Kuramoto model as a gradient flow \cite{V-W} with the analytical potential $V$ on $\bbr^N$:
\begin{align} \label{Ku-grad}
\begin{aligned}
&  {\dot \Theta} = -\nabla_{\Theta} V(\Theta), \\
&  V[\Theta] :=-\sum_{k=1}^N\nu_k \theta_k+ \frac{\kappa}{2N}\sum_{k,l=1}^N \big(1-\cos(\theta_k-\theta_l)\big).
\end{aligned}
\end{align}
Note that the double sum in $\eqref{Ku-grad}_2$ can be simplified using \eqref{R-order}:
\begin{equation} \label{p-o}
 \frac{\kappa}{2N}\sum_{k,l=1}^N \big(1-\cos(\theta_k-\theta_l)\big) =  \frac{\kappa N}{2} (1 -R^2).
 \end{equation}
 Then, it follows from \eqref{Ku-grad} and \eqref{p-o} that 
\[ V[\Theta]  = - \Omega \cdot \Theta  + \frac{\kappa N}{2} (1 -R^2), \]
where $\cdot$ is the usual inner product in $\bbr^N$. The Kuramoto model \eqref{Ku-grad}  has the following property regarding asymptotic dynamics.
\begin{proposition}\label{P2.1}
\emph{\cite{H-K-R}}
Let $\Theta = \Theta(t)$ be a uniformly bounded global solution to \eqref{Ku-grad} in $\bbr^N$:
\begin{equation*}\label{bdd}
\sum_{i=1}^{N} \nu_i = 0 \quad \mbox{and} \quad \sup_{0 \leq t < \infty} \|\Theta(t)\|_{\infty} <\infty.
\end{equation*}
Then, the phase vector $\Theta(t)$ and the frequency vector
 $\dot \Theta(t)$ converge to a phase locked state and the zero vector,
 respectively, as $t\to \infty$, i.e.,  there exists a phase locked state
 $\Theta^{\infty}$ such that
\[ \lim_{t \to \infty} \|\Theta(t) - \Theta^{\infty}\|_{\infty}  = 0 \quad \mbox{and} \quad \lim_{t \to \infty} \|{\dot \Theta}(t)\|_{\infty} = 0. \]
\end{proposition}
\subsection{Previous results} \label{sec:2.2} In this subsection, we briefly review previous results on the emergence of asymptotic phase-locking which is closely related with main results in this paper. 
\begin{theorem} \label{T2.1}
\emph{\cite{C-H-J-K,D-B1}} 
Suppose that natural frequencies, the coupling strength and initial data $\Theta^{in}$ satisfy
\begin{equation*} \label{NE-1}
\sum_{i=1}^{N} \nu_i = 0, \quad  \kappa > \cD(\Omega), \quad   \cD(\Theta^{in}) <  \pi - {{ \arcsin \left (\frac{\cD(\Omega)}{\kappa} \right)}}, 
\end{equation*}
and let $\Theta = \Theta(t)$ be a solution to \eqref{Ku}. Then, the following assertions hold:

\vspace{.1cm}

\begin{enumerate}
\item
The phase diameter is bounded: there exists a finite time $T > 0$ such that
\[
D(\Theta(t))\le \arcsin \left (\frac{\cD(\Omega)}{\kappa} \right),\quad \forall~t\ge T.
\]
\item
The phase vector $\Theta(t)$ approaches a phase-locked state $\Theta^\infty$ with a exponential rate: there exist positive
constants $C_0(T)$ and $\Lambda = {\mathcal O}(\kappa)$ such that
\[ \cD({\dot \Theta}(t)) \leq C_0 \exp(-\Lambda t), \quad \forall~t\ge 0.\]
\item
The emergent phase-locked state $\Theta^\infty$ is unique up to $U(1)$-symmetry, and is ordered according to the ordering of their natural frequencies: there are constants $U$ and $L$ such that for any indices $i,j$ with $\nu_i\ge \nu_j$,
\[
\sin^{-1}\left(\frac{\nu_i-\nu_j}{\kappa U}\right)\le \theta_i^\infty-\theta_j^\infty \le \sin^{-1}\left(\frac{\nu_i-\nu_j}{\kappa L}\right).
\]
\end{enumerate}
\end{theorem}
\begin{remark}
Asymptotic phase-locking for generic initial data was first obtained by Ha et al. \cite{H-K-R} in a sufficiently large coupling regime. An important piece of the argument of \cite{H-K-R} is that some statements of Theorem \ref{T2.1} can be extended to the case where we have only a majority, not the totality, of the population lying on a small arc.
\end{remark}
\begin{theorem}\label{T2.2}
\emph{\cite{H-K-R}}
Suppose the initial configuration $\Theta^{in}$ and coupling strength satisfy
\begin{equation*}\label{generic}
R^{in} >0 ~\mbox{ and }~ \theta_i^{in} \not\equiv \theta_j^{in} \mod 2\pi, \quad \forall~i \neq j, \quad \kappa \gg {\mathcal D}(\Omega),
\end{equation*}
and let $\Theta(t)$ be a solution to \eqref{Ku} with the initial data $\Theta^{in}$. Then, there exists a phase-locked state $\Theta^{\infty}$ such that 
\[ \lim_{t \to \infty} \|\Theta(t) - \Theta^{\infty}\|_{\infty} = 0. \]
\end{theorem}

\section{A gradient flow formulation} \label{sec:3}
\setcounter{equation}{0}
In this section, we first show that the relativistic Cucker-Smale(RCS) model presented in \cite{AHKS} falls down to a relativistic Kuramoto model \eqref{A-4}, and then we introduce a gradient-like flow framework for later usage.
 
\subsection{From  the RCS model to the RK model} \label{sec:3.1}
In this subsection, we briefly review the derivation of the RK model from the RCS model on the unit circle $\bbs^1$. First, we consider the RCS model on $\bbs^d$:
\begin{equation}\label{C-1}
\begin{cases}
\displaystyle \dot{\bx}_i = \bv_i, \quad \forall~t>0, \quad \forall~i=1,\cdots, N, \vspace{.2cm} \\
\displaystyle (\dot\bv_i+\|\bv_i\|^2\bx_i)\bigg(\Gamma_i\bigg(1+\frac{\Gamma_i}{c^2}\bigg)\bigg)+\bv_i\frac{d}{dt}\bigg(\Gamma_i\bigg(1+\frac{\Gamma_i}{c^2}\bigg)\bigg) \vspace{.2cm} \\
\hspace{1cm} = \displaystyle\frac{\kappa}{N} \sum_{j=1}^N \psi(\bx_i, \bx_j) \left({\bv_j}-\bv_i-\frac{\langle\bx_i,\bv_j\rangle}{(1+\langle\bx_i,\bx_j\rangle)}(\bx_i+\bx_j)\right), \vspace{.2cm} \\
\displaystyle (\bx_i (0), \bv_i(0))=(\bx_i^{in},\bv_i^{in})\in \bbt\mathbb{S}^d\subset \mathbb{R}^{d+1}\times \mathbb{R}^{d+1},
\end{cases}
\end{equation}
where $\psi$ is communicate weight function, and $\Gamma_i$ is the Lorentz factor defined by
\begin{equation*} \label{C-1-1}
\Gamma_i := \frac{1}{\sqrt{1- \frac{\|\bv_i\|^2}{c^2}}}, \quad \forall~i = 1, \cdots, N.
\end{equation*}
Since system \eqref{C-1} is reduced from the Riemannian Cucker-Smale model, one has 
\begin{align*}
(\bx_i, \bv_i) \in \bbt\bbs^d, \quad \forall~t\geq0, \quad \forall~i=1,\cdots, N.
\end{align*}
For the derivation of the relativistic Kuramoto model, we choose angle as communication weight function:
\begin{align} \label{C-2}
\psi(\bx_i, \bx_j) = \langle \bx_i, \bx_j \rangle, \quad \forall ~i, j=1, \cdots, N.
\end{align}
Consider $\bbs^1$ embedded in $\bbr^2$, and we set 
\begin{align} \label{C-3}
d = 1, \quad \bx_i := (\cos\theta_{i}, \sin\theta_{i} ), \quad \forall~t\geq0, \quad \forall~i=1,\cdots, N.
\end{align}
This implies
\begin{align} \label{C-4}
\begin{aligned}
& {\bv}_{i} = (-\sin\theta_{i}, \cos\theta_{i}) \dot{\theta}_i, \quad \Gamma_i = \frac{1}{\sqrt{1 - \frac{|{\dot \theta}_i|^2}{c^2}}}, \\
& \dot{\bv}_{i} = -(\cos\theta_{i}, \sin\theta_{i})\dot{\theta}_i^2+(-\sin\theta_{i}, \cos\theta_{i})\ddot{\theta}_i, \quad \forall~t>0, \quad \forall~i=1,\cdots, N.
\end{aligned}
\end{align}
We substitute \eqref{C-2}, \eqref{C-3}, and \eqref{C-4} into \eqref{C-1}$_2$ to obtain
\begin{align} \label{C-5}
\begin{aligned}
\mbox{(LHS)} &= (-\sin\theta_i, \cos\theta_i) \bigg[ \ddot\theta_i \Gamma_i \bigg(1+\frac{\Gamma_i}{c^2}\bigg) +\dot{\theta}_i\frac{d}{dt}\bigg(\Gamma_i\bigg(1+\frac{\Gamma_i}{c^2}\bigg)\bigg) \bigg], \\
\mbox{(RHS)} &= (-\sin\theta_i, \cos\theta_i) \frac{\kappa}{N} \sum_{j=1}^N \cos(\theta_j-\theta_i) \left( \dot\theta_j -\dot\theta_i \right),
\end{aligned}
\end{align}
where we used relation:
\begin{align*}
& \bv_j -\frac{\langle\bx_i,\bv_j\rangle}{(1+\langle\bx_i,\bx_j\rangle)}(\bx_i+\bx_j) \\
& \hspace{.5cm} = (-\sin\theta_{j}, \cos\theta_{j}) \dot{\theta}_{j} -\frac{\dot\theta_j \sin(\theta_i-\theta_j)}{1+\cos(\theta_j-\theta_i)} (\cos\theta_i+\cos\theta_j, \sin\theta_i+\sin\theta_j) \\
& \hspace{.5cm} = (-\sin\theta_i, \cos\theta_i) \dot\theta_j.
\end{align*}
It follows from \eqref{C-5} that
\begin{align*}
& \frac{\kappa}{N} \sum_{j=1}^{N} \cos(\theta_j-\theta_i) \left( \dot\theta_j -\dot\theta_i \right) \\
& \hspace{.5cm} = \ddot\theta_i \Gamma_i \bigg(1+\frac{\Gamma_i}{c^2}\bigg) +\dot{\theta}_i\frac{d}{dt}\bigg(\Gamma_i\bigg(1+\frac{\Gamma_i}{c^2}\bigg)\bigg) = \frac{d}{dt} \bigg( \dot\theta_i \Gamma_i \bigg(1+\frac{\Gamma_i}{c^2}\bigg) \bigg).
\end{align*}
This yields
\begin{align*}
& \dot\theta_i(t) \Gamma_i(t) \bigg(1+\frac{\Gamma_i(t)}{c^2}\bigg) \\
&\hspace{.2cm} = \dot\theta_i(0) \Gamma_i(0) \bigg(1+\frac{\Gamma_i(0)}{c^2}\bigg) +\frac{\kappa}{N} \sum_{j=1}^N \int_0^t \cos(\theta_j(s) -\theta_i(s)) \left( \dot\theta_j(s) -\dot\theta_i(s) \right) ds \\
& \hspace{0.2cm} = \dot\theta_i(0) \Gamma_i(0) \bigg(1+\frac{\Gamma_i(0)}{c^2}\bigg) -\frac{\kappa}{N} \sum_{j=1}^N \sin(\theta_j(0) -\theta_i(0)) +\frac{\kappa}{N} \sum_{j=1}^N \sin(\theta_j(t) -\theta_i(t)).
\end{align*}
Then, we can introduce natural frequencies depending on initial data:
\begin{align*}
& \nu_i := \dot\theta_i(0) \Gamma_i(0) \bigg(1+\frac{\Gamma_i(0)}{c^2}\bigg) -\frac{\kappa}{N} \sum_{j=1}^N \sin(\theta_j(0) -\theta_i(0)), \quad \forall ~i=1,\cdots,N
\end{align*}
and we obtain the Kuramoto type model:
\begin{align} \label{C-6}
\begin{aligned}
& \dot\theta_i \Gamma_i \bigg(1+\frac{\Gamma_i}{c^2}\bigg) = \nu_i +\frac{\kappa}{N} \sum_{j=1}^N \sin(\theta_j -\theta_i), \quad \forall~t\geq0, \quad \forall~i=1,\cdots, N.
\end{aligned}
\end{align}
Note that the well-definedness of \eqref{C-6} can be followed from the fact that 
\begin{align} \label{C-7}
x ~\mapsto~ \frac{cx}{\sqrt{c^2-x^2}}\bigg( 1+\frac{1}{c\sqrt{c^2-x^2}} \bigg)
\end{align}
is monotone increasing odd function on $(-c, c)$ whose image is $\bbr$. From now on, we call this system \eqref{C-6} as the relativistic Kuramoto model. In what follows, we consider two approximations of the L.H.S. of \eqref{C-7} in a low velocity regime as follows.

\vspace{.2cm}

\noindent  $\bullet$~(The first approximation):~Suppose that $\big|\dot\theta_i\big|$'s are sufficiently small compared to the speed of light $c$:
\[  |{\dot \theta}_i| \ll c, \quad \forall~i = 1, \cdots, N \]
so that L.H.S. of \eqref{C-6} can be approximated as 
\begin{align} \label{C-8}
\dot\theta_i\Gamma_i\bigg(1+\frac{\Gamma_i}{c^2}\bigg) \approx \dot\theta_i\Gamma_i = \frac{ \dot\theta_i }{\sqrt{1 - \frac{|{\dot \theta}_i|^2}{c^2}}} =: F_\ell({\dot \theta}_i),
\end{align}
which is the {\it proper velocity} of the $i$-th oscillator. Note that $F_\ell(x) = \frac{x}{\sqrt{1-\frac{x^2}{c^2}}}$ is also odd and monotonically increasing on $(-c, c)$ just like \eqref{C-7}. Hence, our first approximated system for \eqref{C-6} becomes 
\begin{equation} \label{C-9}
F_\ell({\dot \theta}_i) =  \nu_i +\frac{\kappa}{N} \sum_{j=1}^N \sin(\theta_j -\theta_i),
\end{equation}
where $F_\ell \in {\mathcal C}^1(-c, c ~;\bbr)$ is monotone increasing odd function whose image is $\bbr$. 

\vspace{.2cm}

\noindent  $\bullet$~(The second approximation):~Next, we consider the second approximation of \eqref{C-6}. Let $\phi_i$ be the {\it rapidity} of the $i$-th oscillator:
\begin{align*}
\phi_i := \tanh^{-1} \bigg( \frac{\dot\theta_i}{c} \bigg), \quad \forall~t\geq0.
\end{align*}
Then, the L.H.S. of \eqref{C-6} can be further approximated as follows:
\begin{align*}
\dot\theta_i \Gamma_i \bigg(1+\frac{\Gamma_i}{c^2}\bigg) &\approx \dot\theta_i\Gamma_i = \frac{\dot\theta_i}{\sqrt{1-\Big(\frac{\dot\theta_i^2}{c}\Big)^2}} = \frac{c\tanh\phi_i}{\sqrt{1-\tanh^2\phi_i}} = c\sinh\phi_i \\
& \approx c\phi_i = c\tanh^{-1} \bigg( \frac{\dot\theta_i}{c} \bigg) =: F_{r}({\dot \theta}_i).
\end{align*}
Note that $F_r(x) =  c\tanh^{-1}(x/c)$ is also odd and monotonically increasing on $(-c, c)$. Thus, the second approximation of \eqref{C-6} becomes,
\begin{align} \label{C-10}
F_r \big( \dot\theta_i \big) = \nu_i +\frac{\kappa}{N} \sum_{j=1}^N \sin(\theta_j -\theta_i), \quad \forall ~t\geq0, \quad \forall ~i=1,\cdots, N,
\end{align}
where $F_r \in {\mathcal C}^1(-c, c ~;\bbr)$ is monotone increasing odd function whose image is $\bbr$. 

Finally, we present an energy estimate and close this subsection. Recall that
\begin{equation} \label{C-12}
F \big( \dot\theta_i \big) = \nu_i +\frac{\kappa}{N} \sum_{j=1}^N \sin(\theta_j -\theta_i), \quad \forall~i=1,\cdots,N,
\end{equation}
where $F$ satisfies the structural conditions \eqref{A-2}. Then, we define an energy functional ${\mathcal E}_F$:
\begin{equation} \label{C-13}
{\mathcal E}_{F}(t) := \sum_{i=1}^N \int_1^{\gamma_i(t)} \frac{L^2}{x^3} {\mathcal F}(x) dx, \quad \forall~t\geq0,
\end{equation}
where $\gamma_i$ and ${\mathcal F}$ are defined by
\begin{equation} \label{C-14}
\gamma_i(t) := \frac{1}{\sqrt{1-\frac{\dot\theta_i(t)^2}{L^2}}}, \quad {\mathcal F}(x) :=  F^{\prime} \bigg( L\sqrt{1-\frac{1}{x^2}} \bigg).
\end{equation}
Since the integrand of \eqref{C-13} is nonnegative and $\gamma_i \geq 1$, ${\mathcal E}_F$ is also nonnegative. Next, we show that ${\mathcal E}_F$ satisfies a dissipative estimate. 
\begin{proposition} \label{P3.1}
Let $\theta_i$ be a global smooth solution to \eqref{C-12}. Then, we have
\[
\frac{d{\mathcal E}_F}{dt} = -\frac{\kappa}{2N}\sum_{i, j=1}^N \cos(\theta_j-\theta_i) \big( \dot\theta_i-\dot\theta_j \big)^2, \quad \forall~t > 0.
\]
\end{proposition}
\begin{proof}
We differentiate \eqref{C-12}, multiply $\dot\theta_i$, and sum the resulting relation with respect to $i$ to obtain
\begin{align} 
\begin{aligned} \label{C-15}
&\sum_{i=1}^N F' \big( \dot\theta_i \big) \dot\theta_i \ddot\theta_i = \frac{\kappa}{N}\sum_{i, j=1}^N \cos(\theta_j-\theta_i) \big( \dot\theta_j-\dot\theta_i \big) \dot\theta_i = -\frac{\kappa}{2N}\sum_{i, j=1}^N \cos(\theta_i-\theta_j) \big( \dot\theta_i-\dot\theta_j \big)^2.
\end{aligned}
\end{align}
Now, we claim:
\[  \sum_{i=1}^n F' \big( \dot\theta_i \big) \dot\theta_i \ddot\theta_i  = \frac{d{\mathcal E}_F}{dt}. \]
It follows from \eqref{C-14} and $F^{\prime}(-\omega) = F^{\prime}(\omega)$ that 
\begin{align}
\begin{aligned} \label{C-16}
& {\mathcal F}(\gamma_i) = F^{\prime} \bigg( L\sqrt{1-\frac{1}{\gamma_i^2}} \bigg) = F'\big(\big| \dot\theta_i \big|\big) = F'\big( \dot\theta_i \big), \quad \dot\gamma_i =  \bigg( 1 - \frac{{\dot \theta}_i^2}{L^2} \bigg)^{-\frac{3}{2}}\frac{\dot\theta_i \ddot\theta_i}{L^2} = \gamma_i^3 \frac{\dot\theta_i \ddot\theta_i}{L^2}.
\end{aligned}
\end{align}
Then, we use relations \eqref{C-16} to see
\begin{equation} \label{C-17}
\sum_{i=1}^N F' \big( \dot\theta_i \big) \dot\theta_i \ddot\theta_i = \sum_{i=1}^N  \frac{L^2 \mathcal F( \gamma_i )}{\gamma_i^3} \dot\gamma_i = \frac{d{\mathcal E}_F}{dt}.
\end{equation}
Finally, we combine \eqref{C-15} and \eqref{C-17} to find the desired result.
\end{proof}

\subsection{A gradient-like flow framework} \label{sec:3.2}
In this subsection, we present a gradient-like flow framework which will be used for \eqref{A-5} in Section \ref{sec:4}. We first introduce a useful lemma.

\begin{lemma} \label{L3.2}
\emph{(Barbalat’s lemma \cite{B})}
Let $f : [0, \infty) \to \bbr$ be a continuous function. The the following assertions hold.
\begin{enumerate}
\item
If $f$ is uniformly continuous and satisfies $\int_0^\infty f(t)dt < \infty$, then one has
\[
\lim_{t\to\infty} f(t) = 0.
\]
\item
If $f$ satisfies $\lim_{t \to \infty} f(t) = \alpha \in \bbr$ and $f^{\prime}$ is uniformly continuous, then one has 
\[ \lim_{t \to \infty} f^{\prime}(t) = 0. \]
\end{enumerate}
\end{lemma}

\noindent Let $\varphi$ be a monotonically increasing odd function. Then for such $\varphi$, we define a vector-valued function on $\bbr^N$:
\[ \Phi(x_1, \cdots, x_N) = (\varphi(x_1), \cdots, \varphi(x_N)). \]
Now we consider an autonomous system:
\begin{equation}
\begin{cases} \label{C-18}
\displaystyle \dot X = \Phi \Big( -\nabla_X V(X) +\alpha\mathbf 1 \Big), \quad \forall~t > 0, \\
\displaystyle \alpha\in\bbr, \quad \mathbf 1 = (1, \cdots, 1) \in\bbr^N,
\end{cases}
\end{equation}
where $V$ is a real-valued ${\mathcal C}^1$ potential function such that 
\begin{equation} \label{C-19}
 \mathbf 1 \cdot\nabla_XV = 0. 
\end{equation}
\begin{lemma} \label{L3.3}
Let $X = X(t)\in\bbr^N$ be a solution to \eqref{C-18}-\eqref{C-19}. Then, the following three assertions hold:
\begin{enumerate}
\item $V(X(t))$ is monotonically decreasing in $t$.
\vspace{0.1cm}
\item 
Suppose $V(X(t))$ is bounded below. Then there exists $V^\infty\in\bbr$ such that 
\[ \lim_{t\to\infty} V(X(t)) = V^\infty. \]
\item 
Suppose $\frac{d}{dt} V(X(t))$ is uniformly continuous. Then we have, for all $i, j=1,\cdots, N$,
\[
\lim_{t\to\infty} (v_i(t) -v_j(t)) \big[ \varphi(v_i(t)-\alpha) -\varphi(v_j(t)-\alpha) \big] = 0,
\]
where $\nabla_X V(X(t)) = (v_1(t), \cdots, v_n(t))$.
\end{enumerate} 
\end{lemma}
\begin{proof}
\noindent (i)~We set 
\[ v_i := \frac{\partial V}{\partial x_i}, \quad \forall~i = 1, \cdots, N. \]
Then, it follows from \eqref{C-19} that 
\begin{equation} \label{C-20}
 \sum_{j=1}^{N} v_j = 0.
 \end{equation}
Now we use \eqref{C-18} and \eqref{C-20} to find 
\begin{align} \label{C-21}
\begin{aligned}
\frac{d}{dt} V(X) &= \nabla_X V(X) \cdot {\dot X} = \nabla_X V(X) \cdot \Phi \Big(  -\nabla_X V(X) +\alpha\mathbf 1 \Big) \\
&= \sum_{i=1}^N  v_i \varphi (-v_i +\alpha) = \sum_{i=1}^N (v_i -\alpha) \varphi(-v_i +\alpha) + \alpha \sum_{i=1}^N \varphi(-v_i +\alpha) \\
&= -\sum_{i=1}^N (v_i -\alpha) \varphi(v_i -\alpha) +\frac{1}{N} \sum_{j=1}^N  (v_j-\alpha) \sum_{i=1}^N \varphi(v_i -\alpha) \\
& = -\frac{1}{2N} \sum_{i, j=1}^N (v_i -v_j)  \Big( \varphi(v_i -\alpha) -\varphi(v_j-\alpha) \Big) \leq 0,
\end{aligned}
\end{align}
which yields our first desired result.

\vspace{.2cm}

\noindent (ii)~Since $V(X)$ is nonincreasing and bounded below, there exists $V^{\infty}$ such that 
\[ V^{\infty} = \lim_{t \to \infty} V(X(t)). \]

\noindent (iii)~Suppose $\frac{d}{dt} V(X(t))$ is uniformly continuous.  Then, we apply Lemma \ref{L3.2} to get 
\begin{align*}
\lim_{t\to\infty} \frac{d}{dt} V(X(t)) = 0.
\end{align*}
Note that since $\varphi$ is monotonically increasing, we have
\begin{align*}
(v_i(t) -v_j(t)) [\varphi(v_i(t)-\alpha) -\varphi(v_j(t)-\alpha)] \geq 0, \quad \forall ~i, j=1,\cdots, N.
\end{align*}
This and \eqref{C-21} imply the desired estimate:
\begin{align*}
\lim_{t\to\infty} (v_i(t) -v_j(t)) [\varphi(v_i(t)-\alpha) -\varphi(v_j(t)-\alpha)] = 0, \quad \forall ~i, j=1,\cdots,N.
\end{align*}
\end{proof}

\section{Emergent collective dynamics} \label{sec:4}
\setcounter{equation}{0}
In this section, we study emergent properties for the following Cauchy problem:
\begin{align*}
\begin{cases}
\displaystyle F\big( \dot\theta_i \big) = \nu_i +\frac{\kappa}{N} \sum_{j=1}^N \sin(\theta_j -\theta_i), \quad \forall~t>0, \\
\displaystyle \theta_i(0) = \theta_i^{in}, \quad \forall~i=1,\cdots, N,
\end{cases}
\end{align*}
where $\nu_i$ is natural frequency, $\kappa$ is coupling strength, and $F\in {\mathcal C}^1(-c, c~;\bbr)$ satisfies the structural conditions \eqref{A-2} with image $\bbr$. Then, one can define $G := F^{-1}:\bbr\longrightarrow(-c, c)$, which is also monotonically increasing odd function, and system \eqref{C-1}$_1$ can also be written as
\begin{align*}
\dot\theta_i = G\bigg( \nu_i +\frac{\kappa}{N} \sum_{j=1}^N \sin(\theta_j -\theta_i) \bigg).
\end{align*}
Since $F'$ is strictly positive, $G$ is also continuously differentiable.

\subsection{A homogeneous ensemble} \label{sec:4.1}
Consider a homogeneous RK ensemble consisting of identical oscillators with the same natural frequency:
\begin{align*}
\nu_i = \nu, \quad \forall ~i=1,\cdots, N.
\end{align*}
Then, the dynamics of the homogeneous ensemble $\Theta = (\theta_1, \cdots, \theta_n)$ is governed by
\begin{align*}
\begin{cases}
\displaystyle F\big( \dot\theta_i \big) = \nu +\frac{\kappa}{N} \sum_{j=1}^N \sin(\theta_j -\theta_i), \quad \forall~t>0, \\
\displaystyle \theta_i(0) = \theta_i^{in}, \quad \forall~i=1,\cdots, N,
\end{cases}
\end{align*}
or equivalently
\begin{align} \label{D-4}
\begin{cases}
\displaystyle \dot\theta_i = G\bigg( \nu +\frac{\kappa}{N} \sum_{j=1}^N \sin(\theta_j -\theta_i) \bigg), \quad \forall~t>0, \\
\displaystyle \theta_i(0) = \theta_i^{in}, \quad \forall~i=1,\cdots,N.
\end{cases}
\end{align}

\subsubsection{A differential inequatlity approach} \label{sec:4.1.1}
Our first result is concerned with the exponential synchronization of a homogeneous Kuramoto ensemble to the one-point cluster configuration following two steps:
\begin{itemize}
\item Step A (Half circle configuration is an invariant set): by continuity argument and structural conditions of $G$ and sinusoidal interactions, one has 
\[   {\mathcal D}(\Theta^{in}) < \pi \quad \Longrightarrow \quad {\mathcal D}(\Theta(t)) \leq{\mathcal D}(\Theta^{in}). \]

\vspace{.2cm}

\item Step B (Derivation of exponential synchronization): we derive a Gronwall's inequality for the phase diameter ${\mathcal D}(\Theta)$:
\[
\frac{d}{dt} {\mathcal D}(\Theta) \leq -\frac{\kappa\sin({\mathcal D}(\Theta^{in}))}{{\mathcal D}(\Theta^{in})} \Big( \min_{\nu-\kappa\leq\omega \leq\nu+\kappa}G'(\omega) \Big) {\mathcal D}(\Theta), \quad \forall ~t>0.
\]
This yields exponential convergence of ${\mathcal D}(\Theta)$ toward zero.
\end{itemize}

\begin{lemma} \label{L4.1}
Suppose initial data and coupling strength satisfy
\[   {\mathcal D}(\Theta^{in}) < \pi \quad \mbox{and} \quad \kappa > 0, \]
and let $\Theta = \Theta(t)$ be a global smooth solution to \eqref{D-4}. Then, one has
\begin{align*}
{\mathcal D}(\Theta(t)) \leq{\mathcal D}(\Theta^{in}), \quad \forall ~t>0.
\end{align*}
\end{lemma}
\begin{proof}
For a proof, we use continuous induction and we study the evolution of the diameter ${\mathcal D}(\Theta)$. For this, we define a set ${\mathcal T}$ and its supremum $T^*$:
\begin{align*}
& {\mathcal T}:= \{ T \in (0, \infty]~:~{\mathcal D}(\Theta(t)) < \pi, \quad \forall ~t \in [0, T) \}, \quad T^*: =  \sup{\mathcal T}.
\end{align*}
By the assumption on ${\mathcal D}(\Theta^{in})$ and the continuity of phase diameter, there exists $\delta > 0$ such that 
\[ {\mathcal D}(\Theta(t)) < \pi, \quad \forall~t \in [0, \delta). \]
Hence, the set ${\mathcal T}$ is nonempty. Now, we claim:
\[ T^* = \infty. \]
Suppose not, i.e., $T^* < \infty$ and one has
\[  {\mathcal D}(\Theta(t))  < \pi, \quad \forall~t \in [0, T^*). \]
Next, we introduce time-dependent extremal indices $M$ and $m$ such that 
\begin{align*}
\theta_M = \max_{1\leq i \leq N} \theta_i \quad \mbox{and} \quad \theta_m = \min_{1\leq i \leq N} \theta_i.
\end{align*}
For $t\in(0, T^*)$, we use \eqref{D-4} to find
\begin{align} \label{D-5}
\begin{aligned}
\frac{d}{dt} {\mathcal D}(\Theta) &= \dot\theta_M -\dot\theta_m \\
&= G\bigg( \nu +\frac{\kappa}{N} \sum_{j=1}^N \sin(\theta_j -\theta_M) \bigg) -G\bigg( \nu +\frac{\kappa}{N} \sum_{j=1}^N \sin(\theta_j -\theta_m) \bigg) \\
&\leq -\Big( \min_{\nu-\kappa\leq \omega \leq\nu+\kappa}G'(\omega) \Big) \frac{\kappa}{N} \sum_{j=1}^N (\sin(\theta_M-\theta_j)+\sin(\theta_j -\theta_m)) \\
&< -\Big( \min_{\nu-\kappa\leq\omega \leq\nu+\kappa}G'(\omega) \Big) \frac{\kappa}{N} \sum_{j=1}^N \sin(\theta_M-\theta_m) \\
&= -\kappa \Big( \min_{\nu-\kappa\leq \omega \leq\nu+\kappa}G'(\omega) \Big) \sin({\mathcal D}(\Theta)),
\end{aligned}
\end{align}
where we used the fact that
\begin{align*}
\sin x+\sin y > \sin(x+y), \quad \forall ~x, y\in(0, \pi).
\end{align*}
Then in \eqref{D-5}, we use  ${\mathcal D}(\Theta) \in (0, \pi)$ and $G^{\prime} \geq 0$ to see that 
\[ \frac{d}{dt} {\mathcal D}(\Theta)  \leq 0, \quad \forall~t \in (0, T^*). \]
This yields
\[ {\mathcal D}(\Theta(t)) \leq  {\mathcal D}(\Theta^{in}) < \pi, \quad \forall~t \in (0, T^*). \]
Therefore, there exists a positive constant $\delta^{\prime}$ such that 
\[ {\mathcal D}(\Theta(t))  < \pi, \quad \forall~t \in [0, T^* + \delta^{\prime}), \]
which contradicts to the definition of $T^*$. Hence $T^* = \infty$ and we have
\begin{align*}
{\mathcal D}(\Theta(t)) \leq {\mathcal D}(\Theta^{in}),   \quad  \forall~t \in [0, \infty).
\end{align*}
\end{proof}

Now, we are ready to present an exponential synchronization for restricted initial data lying in a half circle. 
\begin{theorem} \label{T4.1}
Suppose that initial data and coupling strength satisfy
\[   {\mathcal D}(\Theta^{in}) < \pi \quad \mbox{and} \quad \kappa > 0, \]
and let $\Theta = \Theta(t)$ be a global smooth solution to \eqref{D-4}. Then, complete synchronization occurs exponentially fast, i.e., there exists a positive constant $\Lambda = \Lambda(\Theta^{in}, \nu, \kappa, G^{\prime})$ such that  
\[ {\mathcal D}(\Theta(t)) \leq e^{-\Lambda t} {\mathcal D}(\Theta^{in}), \quad \forall~t > 0. \]
\end{theorem}
\begin{proof}
We use \eqref{D-5} and
\[ \sin x \leq \frac{\sin({\mathcal D}(\Theta^{in}))}{{\mathcal D}(\Theta^{in})} x, \quad \forall~x \in [0, {\mathcal D}(\Theta^{in})] \]
to obtain
\begin{align*}
\frac{d}{dt} {\mathcal D}(\Theta) &\leq -\kappa \Big( \min_{\nu-\kappa\leq \omega \leq\nu+\kappa}G'(\omega) \Big) \sin({\mathcal D}(\Theta)) \\
&\leq -\frac{\kappa\sin({\mathcal D}(\Theta^{in}))}{{\mathcal D}(\Theta^{in})} \Big( \min_{\nu-\kappa\leq\omega \leq\nu+\kappa}G'(\omega) \Big) {\mathcal D}(\Theta), \quad \forall~t>0.
\end{align*}
Then, Gr\"onwall's lemma yields the desired estimate:
\begin{align*}
& {\mathcal D}(\Theta(t)) \leq e^{-\Lambda t}{\mathcal  D}(\Theta^{in}), \quad \Lambda := \frac{\kappa\sin({\mathcal D}(\Theta^{in}))}{{\mathcal D}(\Theta^{in})} \Big( \min_{\nu-\kappa\leq\omega \leq\nu+\kappa}G'(\omega) \Big) > 0,
\end{align*}
\end{proof}

\subsubsection{A gradient flow approach} \label{sec:4.1.2}
In this part, we extend a formation of complete synchronization for generic initial data using a gradient-like formation of \eqref{D-4}. Motivated by the gradient flow formulation of the Kuramoto model, the potential function $V$ can be written as
\[ V(\Theta) := -\frac{\kappa}{2N} \sum_{i, j=1}^N \cos(\theta_i -\theta_j) \]
and for all $i= 1, \cdots, N$,
\begin{align*}
\dot\theta_i = G\bigg( \nu +\frac{\kappa}{N} \sum_{j=1}^N \sin(\theta_j -\theta_i) \bigg) \Longleftrightarrow \dot\theta_i = G \bigg( \nu -\frac{\partial V}{\partial \theta_i}     \bigg).
\end{align*}
In order to use Lemma \ref{L3.3}, we set 
\begin{align} \label{D-5-2}
\begin{aligned}
& X(t) := \Theta(t), \quad \alpha := \nu, \quad \varphi := G, \quad V(\Theta) := -\frac{\kappa}{2N} \sum_{i, j=1}^N \cos(\theta_i -\theta_j).
\end{aligned}
\end{align}
As an application of Lemma \ref{L3.3}, one has the following complete synchronization. 
\begin{lemma} \label{L4.2}
Let $\Theta = \Theta(t)$ be a global smooth solution to \eqref{D-4}. Then, we have, for all $t>0$ and $i=1,\cdots,N$,
\[ \frac{d}{dt} \sum_{i, j=1}^N \cos(\theta_i -\theta_j) \geq 0, \quad \lim_{t\to\infty} \sum_{j=1}^n \sin(\theta_i-\theta_j) = 0.
\]
\end{lemma}
\begin{proof}
(i) (First relation): By the setting \eqref{D-5-2}, in order to use \eqref{L3.3},  it sufficies to check 
\[
\sum_{i=1}^N \frac{\partial V}{\partial \theta_i} = \frac{\kappa}{N} \sum_{i, j=1}^N \sin(\theta_i -\theta_j) = 0.
\]
We apply Lemma \ref{L3.3} (1)  to see
\begin{equation} \label{D-5-3}
 \frac{d}{dt} V(\Theta(t))=   -\frac{\kappa}{2N} \frac{d}{dt} \sum_{i, j=1}^N \cos(\theta_i -\theta_j) \leq 0,  
\end{equation}
which yields the first desired result.

\vspace{.2cm}

\noindent (ii) (Second relation):~Note that ${\dot \theta}_i$ and $V(\Theta(t))$ have uniform-in-time lower bound:
\begin{align}
\begin{aligned} \label{D-5-4}
& \sup_{0\leq t < \infty} |{\dot \theta}_i(t)| \leq \max \Big \{ G(\nu -\kappa),~G(\nu + \kappa) \Big \}, \\
& \sup_{0\leq t < \infty} \Big\{ |V(\Theta(t))|,~ \Big |\frac{d}{dt} V(\Theta(t)) \Big | \Big \} \leq \frac{\kappa N}{2}.
\end{aligned}
\end{align}
We differentiate \eqref{D-4} with respect to $t$ to get 
\[
{\ddot \theta}_i = G^{\prime} \bigg( \frac{\kappa}{N} \sum_{j=1}^N \sin(\theta_j -\theta_i) \bigg) \cdot \Big(   \frac{\kappa}{N} \sum_{j=1}^N \cos(\theta_j -\theta_i)  ({\dot \theta}_j - {\dot \theta}_i)  \Big).
\]
This yields
\begin{equation} \label{D-5-5}
\sup_{0\leq t < \infty} |{\ddot \theta}_i(t)| \leq \kappa \Big( \max_{-\kappa \leq \omega \leq \kappa} G^{\prime}(\omega) \Big) \cdot \max \Big \{ G(\kappa),~G(-\kappa) \Big \}. 
\end{equation}
On the other hand, it follows from \eqref{D-5-3} that  
\begin{align} \label{D-5-6}
\begin{aligned}
& \frac{d^2}{dt^2} V(\Theta(t)) = \frac{\kappa}{2N} \sum_{i, j=1}^N \Big[ \cos(\theta_i-\theta_j) (\dot\theta_i-\dot\theta_j)^2 +\sin(\theta_i-\theta_j)(\ddot\theta_i -\ddot\theta_j) \Big].
\end{aligned}
\end{align}
In \eqref{D-5-6}, since $\dot\theta_i$ and $\ddot\theta_i$ are uniformly bounded by \eqref{D-5-4} and \eqref{D-5-5}, we have
\[ \sup_{0 \leq t < \infty} \Big| \frac{d^2}{dt^2} V(\Theta(t)) \Big| < \infty. \]
This implies that  $\frac{d}{dt} V(\Theta(t))$ is uniformly continuous and then, by Lemma \ref{L3.2} and Lemma \ref{L3.3}, one has
\begin{align*}
\lim_{t\to\infty}(v_i(t)-v_j(t))(G(v_i(t)-\nu) -G(v_j(t)-\nu)) = 0, \quad \forall ~i, j=1,\cdots, N,
\end{align*}
where
\begin{align*}
v_i := \frac{\kappa}{N} \sum_{k=1}^N \sin(\theta_i -\theta_k), \quad \forall~i=1,\cdots, N.
\end{align*}
On the other hand, since
\begin{align*}
& |v_i -v_j|\cdot|G(v_i-\nu) -G(v_j-\nu)| \geq \bigg( \min_{-\nu-\kappa\leq x\leq -\nu+\kappa} G'(x)\bigg) |v_i -v_j|^2,
\end{align*}
we obtain
\begin{align*}
\lim_{t\to\infty} |v_i -v_j| = 0.
\end{align*}
Finally, it follows from $\sum_{i=1}^N v_i = 0$ that
\begin{align*}
\lim_{t\to\infty} v_i = \lim_{t\to\infty} \frac{1}{N} \sum_{j=1}^N (v_i-v_j) = 0,
\end{align*}
which yields our desired result.
\end{proof}

Next, we analyze the order parameters. For a global solution $\Theta$ to \eqref{D-4}, we introduce order paramerters $(R, \phi)$ by the following relation:
\begin{align} \label{D-6}
Re^{\mathrm i\phi} = \frac{1}{N} \sum_{j=1}^N e^{\mathrm i\theta_j}.
\end{align}
If $R$ is strictly positive for some time interval ${\mathcal I}$, then $\phi$ can be defined smoothly on ${\mathcal I}$. Note that 
\begin{align*}
R^2 &= \frac{1}{N^2} \bigg( \sum_{i=1}^N e^{\mathrm i\theta_i} \bigg) \bigg( \sum_{j=1}^N e^{-\mathrm i\theta_j} \bigg) = \frac{1}{N^2} \sum_{i, j=1}^N e^{\mathrm i(\theta_i -\theta_j)} = \frac{1}{N^2} \sum_{i, j=1}^n \cos(\theta_i -\theta_j),
\end{align*}
and it follows from Lemma \ref{L4.2} that $R$ is monotonically increasing. Hence, if $R(0) > 0$, order parameters are well defined for all $t\geq0$. Moreover, since $R\in(0, 1]$, there exists $R^\infty > 0$ such that
\begin{align*}
\lim_{t\to\infty}R(t) = R^\infty.
\end{align*}
Furthermore, we divide both sides of \eqref{D-6} by $e^{\mathrm i\theta_i}$ and take imaginary part to obtain
\begin{align*}
R\sin(\phi-\theta_i) = \frac{1}{N} \sum_{j=1}^N \sin(\theta_j -\theta_i).
\end{align*}
By Lemma \ref{L3.3} (ii), we have
\begin{align*}
\lim_{t\to\infty} \sin(\phi-\theta_i) = 0, \quad \forall~i=1,\cdots, N.
\end{align*}
This implies
\begin{align*}
& \lim_{t\to\infty} \sin(\theta_i-\theta_j) = \lim_{t\to\infty} \Big(  \sin(\theta_i-\phi)\cos(\theta_j-\phi)-\cos(\theta_i-\phi)\sin(\theta_j-\phi) \Big) = 0.
\end{align*}
We summarize those arguments in the following theorem.
\begin{theorem} \label{T4.2}
Suppose that initial data and coupling strength satisfy
\begin{align*}
R(0) = \bigg| \frac{1}{N} \sum_{i=1}^N e^{\mathrm i\theta_i^{in}} \bigg| > 0, \quad \kappa > 0,
\end{align*}
and let $\theta_i$ be a global solution of \eqref{D-4}. Then, we have
\begin{align*}
\lim_{t\to\infty} (\theta_i -\theta_j) = 0 \mod ~\pi, \quad \forall~i, j=1,\cdots, N.
\end{align*}
\end{theorem}
\begin{remark}
Note that Theorem \ref{T4.2} guarantees complete synchronization and it guarantees dichotomy: complete synchronization or bipolar state. 
\end{remark}

\subsection{A heterogeneous ensemble} \label{sec:4.2}
Consider the generalized Kuramoto model with distributed natural frequencies:
\begin{align} \label{D-7}
\begin{cases}
\displaystyle \dot\theta_i = G\bigg( \nu_i +\frac{\kappa}{N} \sum_{j=1}^N \sin(\theta_j -\theta_i) \bigg), \quad \forall~t>0, \\
\displaystyle \theta_i(0) = \theta_i^{in}, \quad \forall~i=1,\cdots,N.
\end{cases}
\end{align}

\begin{lemma} \label{L4.3}
Suppose that initial data, natural frequency, and coupling strength satisfy
\begin{align*}
& \kappa > \frac{{\mathcal D}(\Omega)}{{\mathcal D}(\Theta^{in})}>0, \quad {\mathcal D}(\Theta^{in}) < \pi-\theta_*, \quad \theta_* :=\sin^{-1} \bigg( \frac{{\mathcal D}(\Omega)}{\kappa} \bigg) \in\bigg( 0, ~\frac{\pi}{2} \bigg),
\end{align*}
and let $\Theta$ be the global smooth solution to \eqref{D-7}. Then, there exists $t_* \geq 0$ such that, for all $t\geq t_*$,
\begin{align*}
D(\Theta(t)) \leq \max\{ \theta_*, ~\min\{ D(\Theta^{in}), ~\pi-D(\Theta^{in}) \}\} \leq \frac{\pi}{2}.
\end{align*}
\end{lemma}
\begin{proof}
We use a similar argument in the proof of Lemma \ref{L4.1}. As in the proof of Lemma \ref{L4.1}, let $M$ and $m$ be time dependent indices such that 
\begin{align*}
\theta_M = \max_{1\leq i \leq N} \theta_i \quad \mbox{and} \quad \theta_m = \min_{1\leq i \leq N} \theta_i,
\end{align*}
and furthermore, we define two constants $\nu_*$ and $\nu^*$ satisfying
\begin{align*}
\nu_* := \min_{1\leq A\leq N} \nu_i \quad \mbox{and} \quad \nu^* := \max_{1\leq A\leq N} \nu_i.
\end{align*}

\noindent $\bullet$ Case A (${\mathcal D}(\Theta^{in}) \leq \theta_*$): Suppose that the following set is nonempty and we set its supremum:
\begin{align*}
{\mathcal T} := \{ t>0:~{\mathcal D}(\Theta(t)) > \theta_* \}, \quad t_* := \inf {\mathcal T}.
\end{align*}
Then, it follows from the definition of $t_*$ that 
\begin{align} \label{D-8}
{\mathcal D}(\Theta(t_*)) = \theta_* \quad \mbox{and} \quad  \frac{d}{dt} \bigg|_{t=t_*+} {\mathcal D}(\Theta) \geq 0.
\end{align}
On the other hand, since
\begin{align*}
& \nu_M +\frac{\kappa}{N} \sum_{j=1}^N \sin(\theta_j(t_*) -\theta_M(t_*)) -\bigg( \nu_m +\frac{\kappa}{N} \sum_{j=1}^N \sin(\theta_j(t_*) -\theta_m(t_*)) \bigg) \\
& \hspace{.3cm} \leq \cD(\Omega) -\frac{\kappa}{N}\sum_{j=1}^N \Big[ \sin(\theta_M(t_*)-\theta_j(t_*)) +\sin(\theta_j(t_*)-\theta_m(t_*)) \Big] \\
& \hspace{.3cm} < \cD(\Omega) -\kappa\sin(\theta_M(t_*) -\theta_m(t_*)) = \cD(\Omega) -\kappa\sin\theta_* = 0,
\end{align*}
one can conclude that at $t = t_*+$,
\begin{align*}
\begin{aligned}
\frac{d}{dt} {\mathcal D}(\Theta) &= G\bigg( \nu_M +\frac{\kappa}{N} \sum_{j=1}^N \sin(\theta_j -\theta_M) \bigg) -G\bigg( \nu_m +\frac{\kappa}{N} \sum_{j=1}^N \sin(\theta_j -\theta_m) \bigg) \\
& \leq \bigg( \min_{\nu_*-\kappa \leq \omega \leq \nu^*+\kappa}G'(\omega) \bigg) \bigg( {\mathcal D}(\Omega) -\frac{\kappa}{N}\sum_{j=1}^N \sin(\theta_M-\theta_j) +\sin(\theta_j-\theta_m) \bigg) < 0.
\end{aligned}
\end{align*}
This contradicts to \eqref{D-8}. Therefore, we can conclude that ${\mathcal T}$ is empty, which is our desired result with $t_* = 0$.

\vspace{.2cm}

\noindent $\bullet$ Case B ($\theta_* < {\mathcal D}(\Theta^{in})$):  It follows from the continuity of $\theta_i$ that there exists $T >0$ satisfying
\begin{align*}
\theta_* < {\mathcal D}(\Theta(t)) < \pi -\theta_*, \quad \forall ~t\in[0, T).
\end{align*}
Then, for all $t\in(0, T)$, we have
\begin{align} \label{D-9}
\begin{aligned}
& \nu_M +\frac{\kappa}{N} \sum_{j=1}^N \sin(\theta_j -\theta_M) -\bigg( \nu_m +\frac{\kappa}{N} \sum_{j=1}^N \sin(\theta_j -\theta_m) \bigg) \\
& \hspace{1cm} \leq \cD(\Omega) -\frac{\kappa}{N}\sum_{j=1}^N \sin(\theta_M -\theta_j) +\sin(\theta_j -\theta_m) \\
& \hspace{1cm} < \cD(\Omega) -\kappa\sin((\Theta)) < \cD(\Omega) -\kappa\sin\theta_* = 0.
\end{aligned}
\end{align}
This implies that for all $t\in(0, T)$,
\begin{align} \label{D-10}
\begin{aligned}
\frac{d}{dt} {\mathcal D}(\Theta) &= G\bigg( \nu_M +\frac{\kappa}{N} \sum_{j=1}^N \sin(\theta_j -\theta_M) \bigg) -G\bigg( \nu_m +\frac{\kappa}{N} \sum_{j=1}^N \sin(\theta_j -\theta_m) \bigg) \\
& \leq \bigg( \min_{\nu_*-\kappa \leq \omega \leq \nu^*+\kappa}G'(\omega) \bigg) \bigg(D(\Omega) -\frac{\kappa}{N}\sum_{j=1}^N \sin(\theta_M-\theta_j) +\sin(\theta_j-\theta_m) \bigg) < 0.
\end{aligned}
\end{align}
Since ${\mathcal D}(\Theta)$ monotonically decrease on $(0, T)$, we can conclude 
\begin{align*}
{\mathcal D}(\Theta(t)) \leq {\mathcal D}(\Theta^{in}) < \pi -\theta_*, \quad \forall ~t\geq0.
\end{align*}

\vspace{.2cm}

\noindent$\diamond$ Case B.1 (${\mathcal D}(\Theta^{in}) \leq \pi/2$): Since
\begin{align*}
\max\{ \theta_*, ~\min\{{\mathcal D}(\Theta^{in}), \pi- {\mathcal D}(\Theta^{in}) \}\} = {\mathcal D}(\Theta^{in}),
\end{align*}
we obtain our desired result with $t_* = 0$.

\vspace{.2cm}

\noindent$\diamond$ Case B.2 (${\mathcal D}(\Theta^{in}) > \pi/2$):~Suppose that
\[  {\mathcal D}(\Theta(t)) > \pi-{\mathcal D}(\Theta^{in}) \quad \mbox{for all $t>0$}. \]
Then, \eqref{D-9} and \eqref{D-10} imply, for all $t>0$,
\begin{align*}
\frac{d}{dt} {\mathcal D}(\Theta) &< \bigg( \min_{\nu_*-\kappa \leq \omega \leq \nu^*+\kappa}G'(\omega) \bigg) ({\mathcal D}(\Omega) -\kappa\sin({\mathcal D}(\Theta^{in})) \\
&<  \bigg( \min_{\nu_*-\kappa \leq \omega \leq \nu^*+\kappa}G'(\omega) \bigg) ({\mathcal D}(\Omega) -\kappa\sin\theta_*) = 0.
\end{align*}
This yields, for all $t>0$,
\begin{align*}
& {\mathcal D}(\Theta(t)) -{\mathcal D}(\Theta^{in}) < t\bigg( \min_{\nu_*-\kappa \leq \omega \leq \nu^*+\kappa}G'(\omega) \bigg) ({\mathcal D}(\Omega) -\kappa\sin({\mathcal D}(\Theta^{in})) < 0.
\end{align*}
Letting $t \to \infty$, we observe ${\mathcal D}(\Theta)$ diverges to $-\infty$, which gives a contradiction. So, we obtain the desired result with
\begin{align*}
t_* \geq \frac{2 {\mathcal D}(\Theta^{in}) -\pi}{\displaystyle \bigg( \min_{\nu_*-\kappa \leq \omega \leq \nu^*+\kappa}G'(\omega) \bigg) (\kappa\sin({\mathcal D}(\Theta^{in}) -{\mathcal D}(\Omega))}.
\end{align*}
\end{proof}
Now, we combine Proposition \ref{P2.1} and Lemma \ref{L3.2} to present the result on the emergence of phase locked state.
\begin{theorem} \label{T4.3}
\emph{(Complete synchronization)}
Suppose that initial data, natural frequency, and coupling strength satisfy
\begin{align*}
& \kappa > \frac{{\mathcal D}(\Omega)}{{\mathcal D}(\Theta^{in})}>0, \quad{\mathcal  D}(\Theta^{in}) < \pi-\theta_*, \quad \theta_* :=\sin^{-1} \bigg( \frac{{\mathcal D}(\Omega)}{\kappa} \bigg) \in\bigg( 0, \frac{\pi}{2} \bigg),
\end{align*}
and let $\theta_i$ be a global smooth solution to \eqref{C-1}. Then, asymptotic complete-frequency synchronization occurs asymptotically:
\begin{align*}
\lim_{t\to\infty} \big| \dot\theta_i(t) -\dot\theta_j(t) \big| = 0.
\end{align*}
\end{theorem}
\begin{proof}
For $t_* \geq 0$, we use Lemma \ref{L4.3} to see,
\begin{align} \label{C-9}
\begin{aligned}
{\mathcal D}(\Theta(t)) &\leq \max\{ \theta_*, ~\min\{{\mathcal D}(\Theta^{in}), \pi-{\mathcal D}(\Theta^{in}) \}\} =: \theta^* < \frac{\pi}{2}, \quad \forall~t\geq t_*,
\end{aligned}
\end{align}
when ${\mathcal D}(\Theta^{in}) \ne \pi/2$. However, when ${\mathcal D}(\Theta^{in}) = \pi/2$, \eqref{C-7} and \eqref{C-8} imply
\begin{align*}
& \frac{d}{dt} \bigg|_{t=0+} {\mathcal D}(\Theta) \leq \bigg( \min_{\nu_*-\kappa \leq \omega \leq \nu^*+\kappa}G'(\omega) \bigg) ({\mathcal D}(\Omega) -\kappa\sin({\mathcal D}(\Theta^{in}))) < 0.
\end{align*}
Without loss of generality, one can assume that ${\mathcal D}(\Theta^{in})\ne\pi/2$. Then, we use \eqref{C-9} and Proposition \ref{P2.1} to obtain
\begin{align*}
& \frac{\kappa}{2N}\sum_{i, j=1}^n \int_{t_*}^t \cos(\theta_j(s)-\theta_i(s)) \big( \dot\theta_i(s)-\dot\theta_j(s) \big)^2 ds \\
&\hspace{.2cm} \leq {\mathcal E}_F(t) +\frac{\kappa}{2n}\sum_{i, j=1}^N \int_{t_*}^t \cos(\theta_j(s)-\theta_i(s)) \big( \dot\theta_i(s)-\dot\theta_j(s) \big)^2 ds = {\mathcal E}_F(t_*)
\end{align*}
This implies 
\begin{align*}
& \cos\theta^*\int_{t_*}^\infty \big( \dot\theta_i(s)-\dot\theta_j(s) \big)^2 ds \leq \int_{t_*}^\infty \cos(\theta_j(s)-\theta_i(s)) \big( \dot\theta_i(s)-\dot\theta_j(s) \big)^2 ds < \infty.
\end{align*}
Now we claim:
\[  \mbox{the map $~t\mapsto\big( \dot\theta_i(t)-\dot\theta_j(t) \big)^2~$ is uniformly continuous.}  \]
Once this claim is verified, we can apply Lemma \ref{L3.2} to obtain the desired result.

\vspace{.2cm}

\noindent {\it Proof of claim}:~It follows from \eqref{C-10} that
\begin{align*}
\big| \dot\theta_i \big| \leq \max\Big\{ \Big| G\Big( \min_{1\leq i \leq N}\nu_i -\kappa \Big) \Big|, ~\Big| G\Big( \max_{1\leq i\leq N}\nu_i +\kappa \Big) \Big| \Big\}, \quad \forall ~t\geq0, \quad \forall ~i=1,\cdots, N.
\end{align*}
We also differentiate \eqref{C-10} and take absolute value to obtain
\begin{align*}
\big| \ddot\theta_i \big| \leq 2\kappa G'(0) \Big( \max_{1\leq i \leq N} \big| \dot\theta_i \big| \Big), \quad \forall~t\geq0, \quad \forall ~i=1,\cdots, N.
\end{align*}
This uniform boundedness implies
\begin{align*}
 \big( \dot\theta_i(t)-\dot\theta_j(t) \big)^2 = 2 \big( \dot\theta_i(t)-\dot\theta_j(t) \big) \big( \ddot\theta_i(t)-\ddot\theta_j(t) \big)
\end{align*}
is also uniformly bounded, from which we can conclude our claim.
\end{proof}
%\begin{align} \label{D-1}
%\begin{cases}
%\displaystyle \dot\theta_i \Gamma_i \bigg(1+\frac{\Gamma_i}{c^2}\bigg) = \nu_i +\frac{\kappa}{n} \sum_{j=1}^n \sin(\theta_j -\theta_i), \vspace{.1cm}\\
%\displaystyle \Gamma_i = \frac{c}{\sqrt{c^2-\dot\theta_i^2}}, \vspace{.1cm}\\
%\displaystyle \theta_i(0) = \theta_i^0, \quad i=1,\cdots,n,
%\end{cases}
%\end{align}
%
%\begin{align*}
%F(\theta) := \frac{c\theta}{\sqrt{c^2-\theta^2}}\bigg( 1+\frac{c}{c^2\sqrt{c^2-\theta^2}} \bigg) = \frac{c\theta}{\sqrt{c^2-\theta^2}}+\frac{\theta}{c^2-\theta^2}, \quad -c<\theta<c
%\end{align*}
%
%\begin{align*}
%F\big( \dot\theta_M \big) -F \big( \dot\theta_m \big) &= \nu_M-\nu_m+\frac{\kappa}{n}\sum_{j=1}^n (\sin(\theta_j-\theta_M) -\sin(\theta_j-\theta_m)) \\
%&\leq \cD(\nu)+\frac{\kappa}{2n}\sum_{j=1}^n \cos\frac{\theta_j-\theta_M+\theta_j-\theta_m}{2} \sin\frac{\theta_m-\theta_M}{2} \\
%&\leq \cD(\nu)-\frac{\kappa}{2}\cos\frac{\cD(\theta)}{2} \sin\frac{\cD(\theta)}{2} \\
%&= \cD(\nu)-\frac{\kappa}{4}\sin\cD(\theta)
%\end{align*}
%
%\begin{align*}
%& \sum_{A, j=1}^n \big( \dot\theta_i-\dot\theta_j \big)\sin(\theta_i-\theta_j) = \sum_{A, j=1}^n 2\dot\theta_i \sin(\theta_i-\theta_j) \\
%& \hspace{.2cm} =2\sum_{A, j=1}^n \sin(\theta_i-\theta_j)G\bigg( \sum_{C=1}^n\sin(\theta_C-\theta_i) \bigg) \\
%& \hspace{.2cm} = -2\sum_{i=1}^n \sum_{j=1}^n \sin(\theta_j-\theta_i)
%\end{align*}

\section{non-relativistic limit of the RK model } \label{sec:5}
\setcounter{equation}{0}
In this section, we study a non-relativistic limit from the relativistic Kuramoto model to the non-relativistic Kuramoto model in any finite-time interval, as $c \to \infty$.  More precisely, we consider the relativistic Kuramoto model: for all $t>0$ and $i=1,\cdots, N$,
\begin{equation}  \label{E-1}
\dot\theta_i \Gamma_i \bigg(1+\frac{\Gamma_i}{c^2}\bigg) = \nu_i +\frac{\kappa}{N} \sum_{j=1}^N \sin(\theta_j -\theta_i),
\end{equation}
and the non-relativistic Kuramoto model:
\begin{equation} \label{E-2}
\dot\theta_i  = \nu_i +\frac{\kappa}{N} \sum_{j=1}^N \sin(\theta_j -\theta_i), \quad \forall~t > 0, \quad \forall~i=1,\cdots,N,
\end{equation}
with the same initial data:
\begin{equation} \label{E-3}
\theta_i(0) = \theta^{in}_i, \quad \forall~i=1,\cdots,N.
\end{equation}
\subsection{A non-relativistic limit} \label{sec:5.1}
In this subsection, we present a non-relativistic limit from \eqref{E-1} to \eqref{E-2} in any finite time interval as $c \to \infty$. We set 
\[  \Theta^c := (\theta_1^c, \cdots, \theta_N^c) \quad \mbox{and} \quad \Theta^{\infty} := (\theta_1^{\infty}, \cdots, \theta_N^{\infty}). \]
Let $\Theta^c$ and $\Theta^\infty$ be the smooth solutions to \eqref{E-1} and \eqref{E-2} with the same initial data \eqref{E-3}, respectively and recall that 
 \begin{equation} \label{E-3-0}
F(x)=\frac{cx}{\sqrt{c^2-x^2}}+\frac{x}{c^2-x^2} \quad \mbox{and} \quad G=F^{-1}. 
\end{equation}
 Then, $\Theta^c$ and $\Theta^{\infty}$ satisfy \eqref{E-1} and \eqref{E-2}:
\begin{equation} \label{E-3-1}
\begin{cases}
\displaystyle \dot\theta^c_i  = G\bigg( \nu_i +\frac{\kappa}{N} \sum_{j=1}^N \sin(\theta^c_j -\theta^c_i) \bigg), \vspace{.1cm} \\
\displaystyle \dot\theta^\infty_i  = \nu_i +\frac{\kappa}{N} \sum_{j=1}^N \sin(\theta^\infty_j -\theta^\infty_i).
\end{cases}
\end{equation}
	\begin{lemma}\label{L5.1}
		Let $\Theta^c$ and $\Theta^\infty$ be smooth solutions to \eqref{E-3-1} with the same initial data \eqref{E-3}. Then, there exists a positive constant $G^{\infty}$ independent of $c$ such that for all $t \geq 0$, 
\begin{enumerate}[(i)]
\item $\displaystyle \Big |G \Big(\nu_i +\frac{\kappa}{N} \sum_{j=1}^N \sin(\theta^c_j -\theta^c_i) \Big)  \Big|\leq G^\infty$,

\item $\displaystyle \Big |G \Big(\nu_i +\frac{\kappa}{N} \sum_{j=1}^N \sin(\theta^c_j -\theta^c_i) \Big) -\Big(\nu_i +\frac{\kappa}{N} \sum_{j=1}^N \sin(\theta^c_j -\theta^c_i) \Big ) \Big |\leq\mathcal{O}(c^{-2})$,
\end{enumerate}
	\end{lemma}
	\begin{proof}
	(i)~We set
	\[  x:= \nu_i +\frac{\kappa}{N} \sum_{j=1}^N \sin(\theta^c_j -\theta^c_i). \] 
	It follows from \eqref{E-3-0} that 
	\[\frac{cG(|x|)}{\sqrt{c^2-G(|x|)^2}}+\frac{G(|x|)}{c^2-G(|x|)^2}=|x|.\] 
	This and the monotonicity of $G$ imply
	\begin{align*}
	& \frac{cG(|x|)}{\sqrt{c^2-G(x)^2}}\leq|x|, \quad |G(x)|=G(|x|)\leq \frac{c|x|}{\sqrt{c^2+|x|^2}}\leq |x|.
	\end{align*}
	This yields
	\begin{align*}
	& \bigg|G\Big(\nu_i +\frac{\kappa}{N} \sum_{j=1}^N \sin(\theta^c_j -\theta^c_i) \Big) \bigg| \leq \bigg|\nu_i +\frac{\kappa}{N} \sum_{j=1}^N \sin(\theta^c_j -\theta^c_i)\bigg| \leq \kappa+\max_{1\leq i\leq N}|\nu_i| =: G^{\infty}.
	\end{align*}
	Note that the constant $G^\infty$  is independent on $c$. \newline
	
	\noindent (ii)~Note that 
	\begin{align*}
	& \frac{cG(x)}{\sqrt{c^2-G(x)^2}}-G(x)+\frac{G(x)}{c^2-G(x)^2}=x-G(x) \iff \frac{cG(x)}{\sqrt{c^2-G(x)^2}}+\frac{G(x)}{c^2-G(x)^2}=x.
	\end{align*}
	This implies
	\begin{align*}
	|G(x)-x| &\leq \bigg|\frac{cG(x)}{\sqrt{c^2-G(x)^2}}-G(x)\bigg|+\bigg|\frac{G(x)}{c^2-G(x)^2}\bigg| \\
	& \leq G^\infty\bigg|\frac{c}{\sqrt{c^2-(G^\infty)^2}}-1\bigg|+\bigg|\frac{G^\infty}{c^2-(G^\infty)^2}\bigg| =\mathcal{O}(c^{-2}),
	\end{align*}
	where we used the relation:
\begin{align*}
& \bigg|\frac{c}{\sqrt{c^2-(G^\infty)^2}}-1\bigg| =\frac{(G^\infty)^2}{(c+\sqrt{c^2-(G^\infty)^2})(\sqrt{c^2-(G^\infty)^2})}.
\end{align*}
	Hence, we have the desired estimate.
	\end{proof}
	\begin{theorem} \label{T5.1}
For $T \in (0, \infty)$, let $\Theta^c$ and $\Theta^\infty$ be two solutions to \eqref{E-3-1} with the same initial data \eqref{E-3}. Then, one has 
\[ \lim_{c\to\infty}\sup_{0\leq t\leq T} \|\Theta^c(t) - \Theta^{\infty}(t) \|_1 = 0. \]	
\end{theorem}
\begin{proof} We split its proof into two steps.

\vspace{.2cm}

\noindent $\bullet$ Step A: We derive the following relation:
\begin{equation} \label{E-3-1-1}
 \frac{d}{dt} \| \Theta^c(t) - \Theta^{\infty}(t) \|_1 \leq \mathcal{O}(c^{-2}), \quad \forall~t > 0.
\end{equation}
It follows from \eqref{E-3-1} that 
\begin{align}
\begin{aligned} \label{E-3-2}
& \frac{1}{2}\frac{d}{dt}|\theta_i^c-\theta_i^\infty|^2 =|\theta_i^c-\theta_i^\infty|\frac{d}{dt}|\theta_i^c-\theta_i^\infty|=(\theta_i^c-\theta_i^\infty)(\dot{\theta}_i^c-\dot{\theta}_i^\infty)\\
&\hspace{.2cm} =(\theta_i^c-\theta_i^\infty)\bigg[ G\Big(\nu_i +\frac{\kappa}{N} \sum_{j=1}^N \sin(\theta^c_j -\theta^c_i)\Big) -\Big(\nu_i +\frac{\kappa}{N} \sum_{j=1}^N \sin(\theta^\infty_j -\theta^\infty_i)\Big)\bigg] \\
&\hspace{.2cm} =(\theta_i^c-\theta_i^\infty)\bigg[ G\Big(\nu_i +\frac{\kappa}{N} \sum_{j=1}^N \sin(\theta^c_j -\theta^c_i)\Big) -\Big(\nu_i +\frac{\kappa}{N} \sum_{j=1}^N \sin(\theta^c_j -\theta^c_i)\Big) \\
&\hspace{3cm} +\Big(\nu_i +\frac{\kappa}{N} \sum_{j=1}^N \sin(\theta^c_j -\theta^c_i)\Big) -\Big(\nu_i +\frac{\kappa}{N} \sum_{j=1}^N \sin(\theta^\infty_j -\theta^\infty_i)\Big)\bigg].
\end{aligned}
\end{align}
Now, we use \eqref{E-3-2}, Lemma \ref{L5.1} and mean-value theorem to find
\begin{align*}
&|\theta_i^c-\theta_i^\infty|\frac{d}{dt}|\theta_i^c-\theta_i^\infty|\\
&\hspace{0.2cm} \leq|\theta_i^c-\theta_i^\infty| \cdot \bigg|G\bigg(\nu_i +\frac{\kappa}{N} \sum_{j=1}^N \sin(\theta^c_j -\theta^c_i)\bigg) -\bigg(\nu_i +\frac{\kappa}{N} \sum_{j=1}^N \sin(\theta^c_j -\theta^c_i)\bigg)\bigg|\\
&\hspace{0.5cm}+(\theta_i^c-\theta_i^\infty)\bigg(\bigg(\nu_i +\frac{\kappa}{N} \sum_{j=1}^N \sin(\theta^c_j -\theta^c_i)\bigg) -\bigg(\nu_i +\frac{\kappa}{N} \sum_{j=1}^N \sin(\theta^\infty_j -\theta^\infty_i)\bigg)\bigg)\\
&\hspace{0.2cm} \leq \mathcal{O}(c^{-2})|\theta_i^c-\theta_i^\infty| +(\theta_i^c-\theta_i^\infty)\bigg(\bigg(\frac{\kappa}{N} \sum_{j=1}^N \sin(\theta^c_j -\theta^c_i)\bigg) -\bigg(\frac{\kappa}{N} \sum_{j=1}^N \sin(\theta^\infty_j -\theta^\infty_i)\bigg)\bigg) \\
&\hspace{0.2cm} =\mathcal{O}(c^{-2})|\theta_i^c-\theta_i^\infty| +\frac{\kappa}{N}(\theta_i^c-\theta_i^\infty)\sum_{j=1}^N\cos(t^{*}_{ij})\left(\theta^c_j -\theta^c_i-\theta^\infty_j +\theta^\infty_i\right),
\end{align*}
where we used $t^{*}_{ij}=t^{*}_{ji}$, the fact that cosine is even, and  the relation:
\[ \min(\theta^c_j -\theta^c_i,\theta^\infty_j -\theta^\infty_i)< t^{*}_{ij}< \max(\theta^c_j -\theta^c_i,\theta^\infty_j -\theta^\infty_i). \]
Then, one has
\begin{align} \label{E-5}
\begin{aligned}
& \frac{d}{dt}|\theta_i^c-\theta_i^\infty| \leq\mathcal{O}(c^{-2})+\frac{\kappa}{N}\cdot \text{sgn}(\theta_i^c-\theta_i^\infty) \sum_{j=1}^N\cos(t^{*}_{ij})\left(\theta^c_j -\theta^c_i-\theta^\infty_j +\theta^\infty_i\right).
\end{aligned}
\end{align}
We take a sum \eqref{E-5} over $i = 1, \cdots, N$ to find the desired estimate:

\begin{align*}
\begin{aligned} 
&\frac{d}{dt} \| \Theta^c(t) - \Theta^{\infty}(t) \|_1 \\
& \hspace{0.2cm} \leq\mathcal{O}(c^{-2}) +\frac{\kappa}{N}\sum_{i,j=1}^N\cos(t^{*}_{ij}) \text{sgn}(\theta_i^c-\theta_i^\infty)\left(\theta^c_j -\theta^c_i-\theta^\infty_j +\theta^\infty_i\right)\\
& \hspace{0.2cm} =\mathcal{O}(c^{-2}) +\frac{\kappa}{2N}\sum_{i,j=1}^N \Big[ \cos(t^{*}_{ij})\left(\text{sgn}(\theta_i^c-\theta_i^\infty)-\text{sgn}(\theta_j^c-\theta_j^\infty)\right)\left(\theta^c_j -\theta^\infty_j-(\theta^c_i -\theta^\infty_i)\right) \Big] \\
& \hspace{0.2cm} \leq \mathcal{O}(c^{-2}),
\end{aligned}
\end{align*}
where we used the standard index interchanging argument and the fact that
\[\left(\text{sgn}(\theta_i^c-\theta_i^\infty)-\text{sgn}(\theta_j^c-\theta_j^\infty)\right)\left(\theta^c_j -\theta^\infty_j-(\theta^c_i -\theta^\infty_i)\right)
\]
is non-positive. 

\vspace{.2cm}

\noindent $\bullet$ Step B: We integrate \eqref{E-3-1-1} over $[0, T]$ using $\|\Theta^{c,in} - \Theta^{\infty, in} \|_1 = 0$ to find
\[  \sup_{0 \leq t \leq T}  \| \Theta^c(t) - \Theta^{\infty}(t) \|_1 \leq  \mathcal{O}(c^{-2}) t \leq  \mathcal{O}(c^{-2}) T.      \]
Letting $c \to \infty$, we obtain the desired estimate.
\end{proof}

\subsection{Numerical simulations} \label{sec:5.2}
In this  subsection, we study the nonlinear response on the emergent dynamics changing the size of $c$ and a non-relativistic limit $(c \to \infty)$ from the relativistic Kurmoto model to the non-relativistic one. \newline

Consider following four Kuramoto type systems:
\begin{equation}
\begin{cases} \label{F-1}
\displaystyle \mbox{Kuramoto model (KM)}: ~{\dot \theta}_i = \nu_i + \frac{\kappa}{N} \sum_{j= 1}^{N} \sin(\theta_j - \theta_i), \\
\displaystyle \mbox{RK model (RKM)}: ~\dot\theta_i \Gamma_i \bigg(1+\frac{\Gamma_i}{c^2}\bigg) = \nu_i +\frac{\kappa}{N} \sum_{j=1}^N \sin(\theta_j -\theta_i), \\
\displaystyle \mbox{Approximate RK model I (ARKM-1)}: ~\frac{ \dot\theta_i }{\sqrt{1 - \frac{|{\dot \theta}_i|^2}{c^2}}}  = \nu_i +\frac{\kappa}{N} \sum_{j=1}^N \sin(\theta_j -\theta_i), \\
\displaystyle \mbox{Approximate RK model II (ARKM-2)}: ~c\tanh^{-1} \bigg( \frac{\dot\theta_i}{c} \bigg) = \nu_i +\frac{\kappa}{N} \sum_{j=1}^N \sin(\theta_j -\theta_i)
\end{cases}
\end{equation}
subject to the same  initial data:
\begin{equation} \label{F-2}
\theta_i(0) = \theta^{in}_i, \quad \forall ~i = 1, \cdots, N.
\end{equation}
%\begin{figure}[h!]
%	\centering
%	\includegraphics[scale=0.4]{traject}
%	\caption{$\Theta^0, \Theta^1,\Theta^2,\Theta^3$ for $c=1$ and $ t=0,3,7 $.}
%	\label{Fig1}
%\end{figure}
For the simplicity of notation, we denote $ \Theta^0, \Theta^{1} ,\Theta^2$ and $\Theta^3$ as solutions to the Kuramoto model, the relativistic Kuramoto model, the approximate Kuramoto model I, and the approximate Kuramoto model II with the same initial data \eqref{F-2}, respectively. All the simulations in what follows, we used the fourth-order Runge-Kutta method. 

\subsubsection{Formation of phase-locked state} \label{sec:5.2.1} In this part, we compare phase-locked states for four systems in \eqref{F-1} issued from the same initial data. For numerical simulations, we choose system parameters as follows:
\begin{equation*}
\Delta t = 0.01,\quad N=10, \quad \kappa=1,
\end{equation*}
and we choose initial data  and natural frequencies randomly from the following intervals, respectively:
\begin{align*}
& \theta^{in}_i \in [-(\pi-\theta_*)/2, ~(\pi-\theta_*)/2], \quad \theta_*:=\sin^{-1} \Big( \frac{{\mathcal D}(\Omega)}{\kappa} \Big), \quad \nu_i \in [-0.15, 0.15], \quad \forall ~i=1,\cdots,N.
\end{align*}

\vspace{-1cm}

\begin{figure}[!h]
\centering
	\includegraphics[scale=0.45]{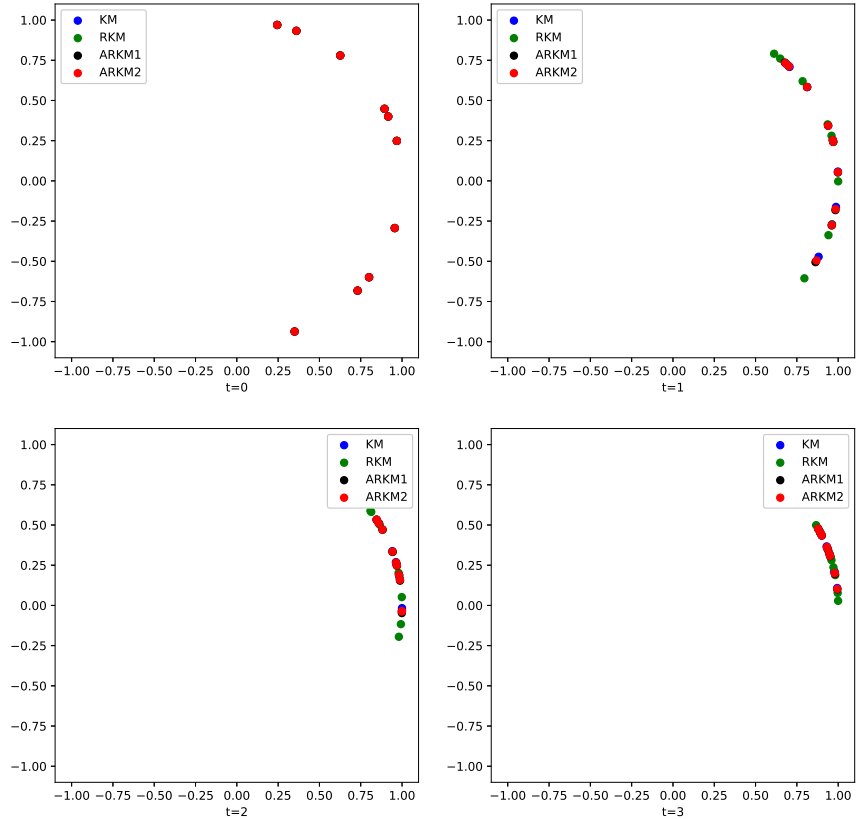}\hfill

	\caption{$\Theta^0, \Theta^1,\Theta^2,\Theta^3$ for $ t=0,1,2,3 $.}
	\label{Fig1}
\end{figure}

In Figure 1, we compare time-evolution of four trajectories corresponding to four different models with the same initial configuration. We plot smooth solutions $ \Theta^0, \Theta^1, \Theta^2, \Theta^3 $ at time $ t=0,1,2,3 $. Note that each flow tends to different phase-locked states exponentially fast, and each model shows different decay rate. The solution of relativistic Kuramoto model $ \Theta^1 $ decay rate is relatively slow compared to that of other models. 

\begin{figure}[!h]
	\begin{subfigure}{0.4\textwidth}
		\centering
		\includegraphics[scale=.5]{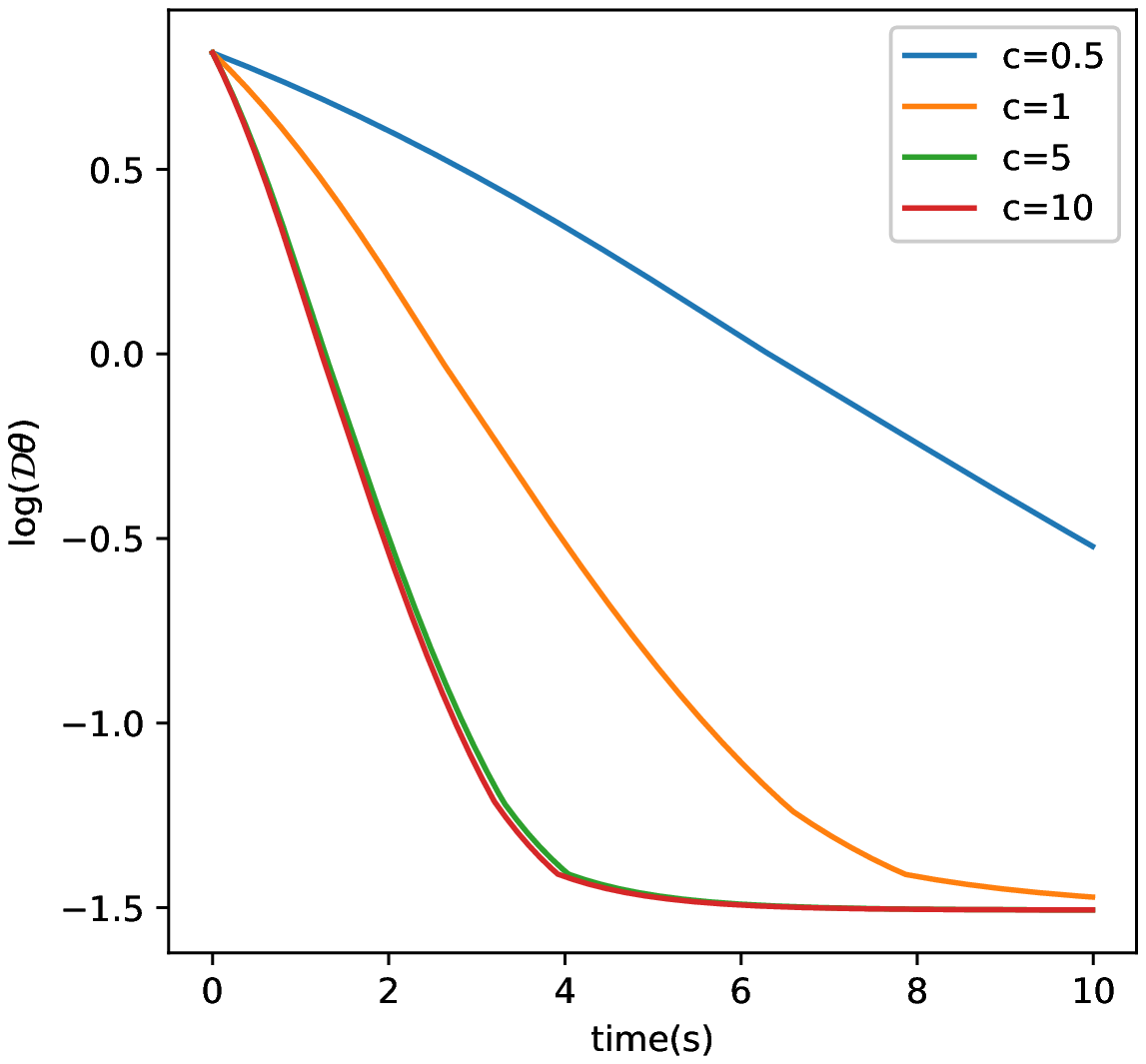}
	\end{subfigure}
	\begin{subfigure}{0.4\textwidth}
	\centering
	\includegraphics[scale=.5]{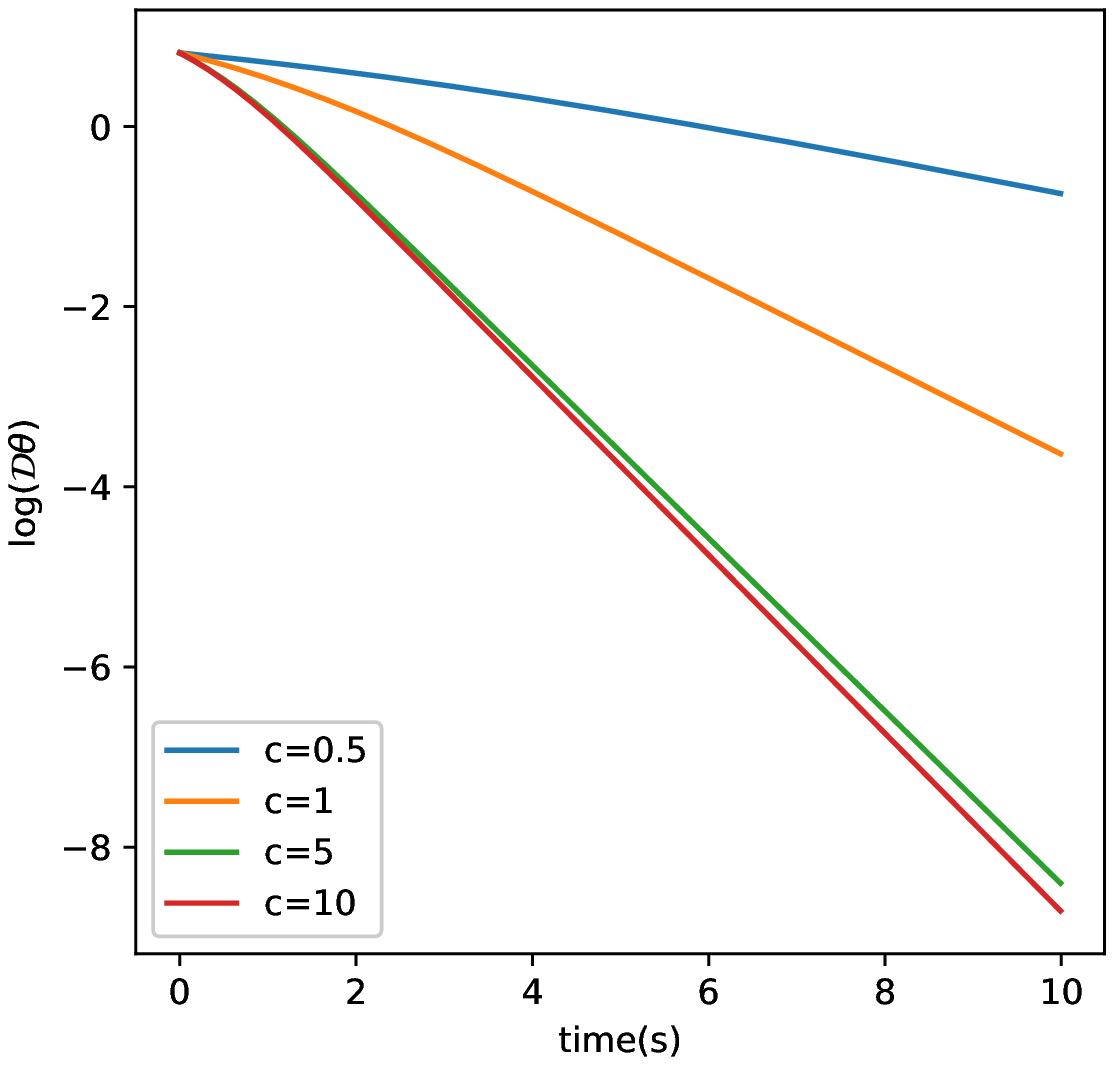}
	\end{subfigure}		
	\caption{time-evolution of $ \log \mathcal{D}({\Theta^1} )$ for $ c=0.5,1,5,10 $.}
	\label{Fig2}
\end{figure}

In Figure 2, we study the decay rate toward the phase-locked state due to the effect of speed of light $c$. The left plot shows the time-evolution of $ \log\mathcal{D}(\Theta) $ in a heterogeneous ensemble, and right plot is for the homogeneous ensemble: $ \nu_i=0 $. Each simulation was conducted for $ c = 0.5 ,1,5,10 $. For the heterogenous case, $  \Theta^1$ converges to phase-locked state with $ \mathcal{D}(\Theta^1)>0$, so $ \log D(\Theta^1) $ converges to certain value, whereas for the homogeneous case, $ \Theta^1 $ converges to a single point, so we can observe the exponential decay of $ \mathcal D(\Theta^1) $. 

The decay rate increases, as $ c $ increases and seems to converge to certain value. Note that the plot when $ c=5 $ and $ c=10 $ is almost identical so that the convergence of decay rate tends to that of the Kuramoto model, as $ c\to \infty$.

%\begin{figure}[h!]
%	\begin{subfigure}{0.4\textwidth}
%		\centering
%		\includegraphics[width=.9\linewidth]{sync}
%		\caption{phase-locking state $ c=1 $}
%		\label{Fig2}
%	\end{subfigure}%
%	\begin{subfigure}{0.4\textwidth}
%		\centering
%		\includegraphics[width=.9\linewidth]{sync_magnified}
%		\caption{magnified phase-locking state $ c=1 $}
%		\label{Fig3}
%	\end{subfigure}
%	\medskip
%	\begin{subfigure}{0.4\textwidth}
%		\centering
%		\includegraphics[width=.9\linewidth]{sync_c}
%		\caption{phase-locking state $ c=5 $}
%		\label{Fig4}
%	\end{subfigure}
%	\begin{subfigure}{0.4\textwidth}
%		\centering
%		\includegraphics[width=.9\linewidth]{sync_mag_c}
%		\caption{magnified phase-locking state $ c=5 $}
%		\label{Fig5}
%	\end{subfigure}
%	\caption{Phase-Locked state}
%	\label{fig:images}
%\end{figure}

\subsubsection{A non-relativistic limit} \label{sec:5.2.2}
In this part, we perform a numerical study on the non-relativistic limit and compare them with analytical results in Theorem \ref{T5.1}. Again, we choose system parameters as follows:
\[
\Delta t = 0.01,\quad N=10, \quad \kappa=1,
\]
and natural frequencies and initial data are randomly chosen from the intervals $[-0.15, 0.15]$ and $[-(\pi-\theta_*)/2,(\pi-\theta_*)/2]$, respectively.

\begin{figure}[!h]
	\begin{subfigure}{0.3\textwidth}
		\centering
		\includegraphics[scale=.4]{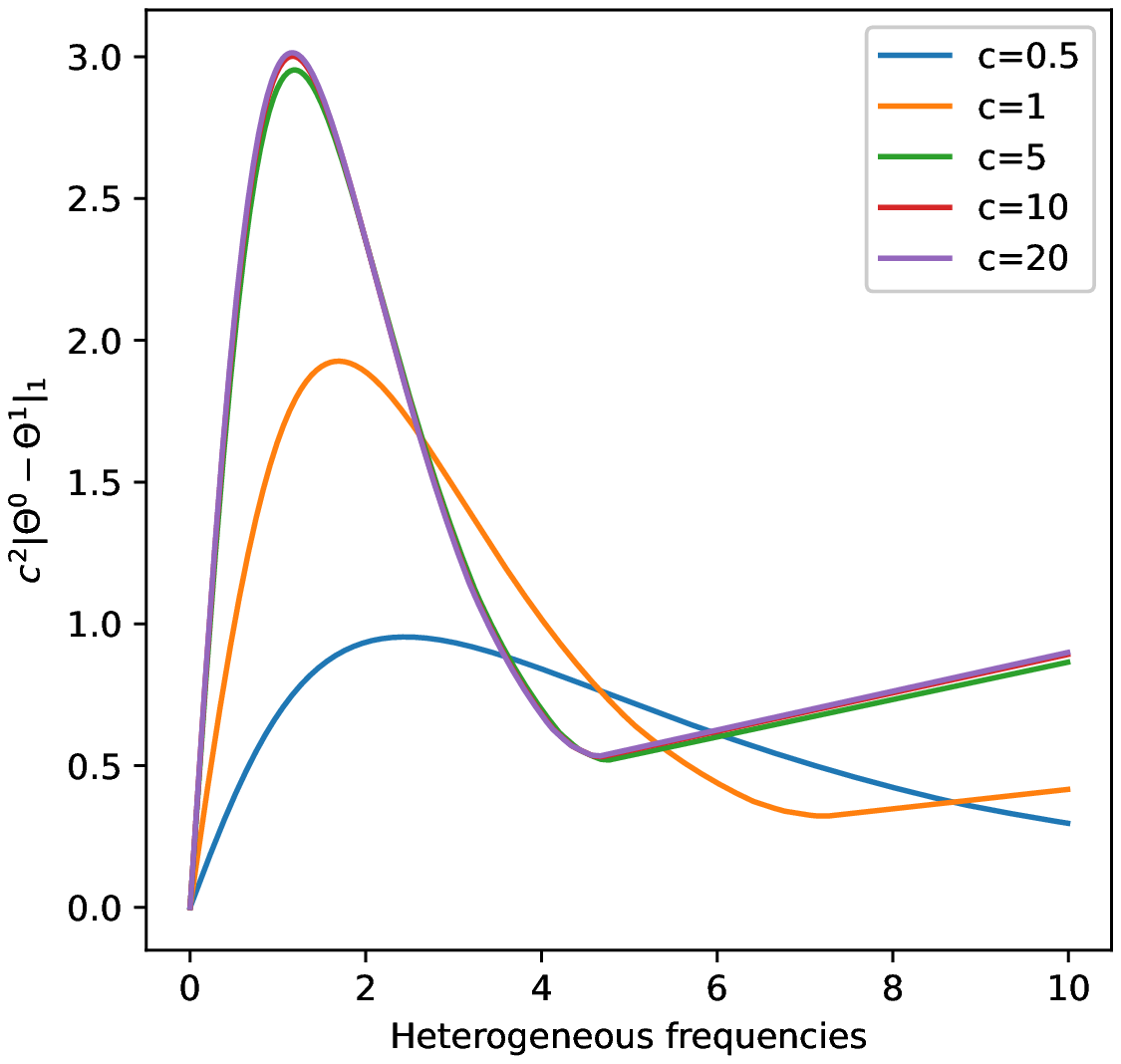}
	\end{subfigure}
	\begin{subfigure}{0.3\textwidth}
		\centering
		\includegraphics[scale=.4]{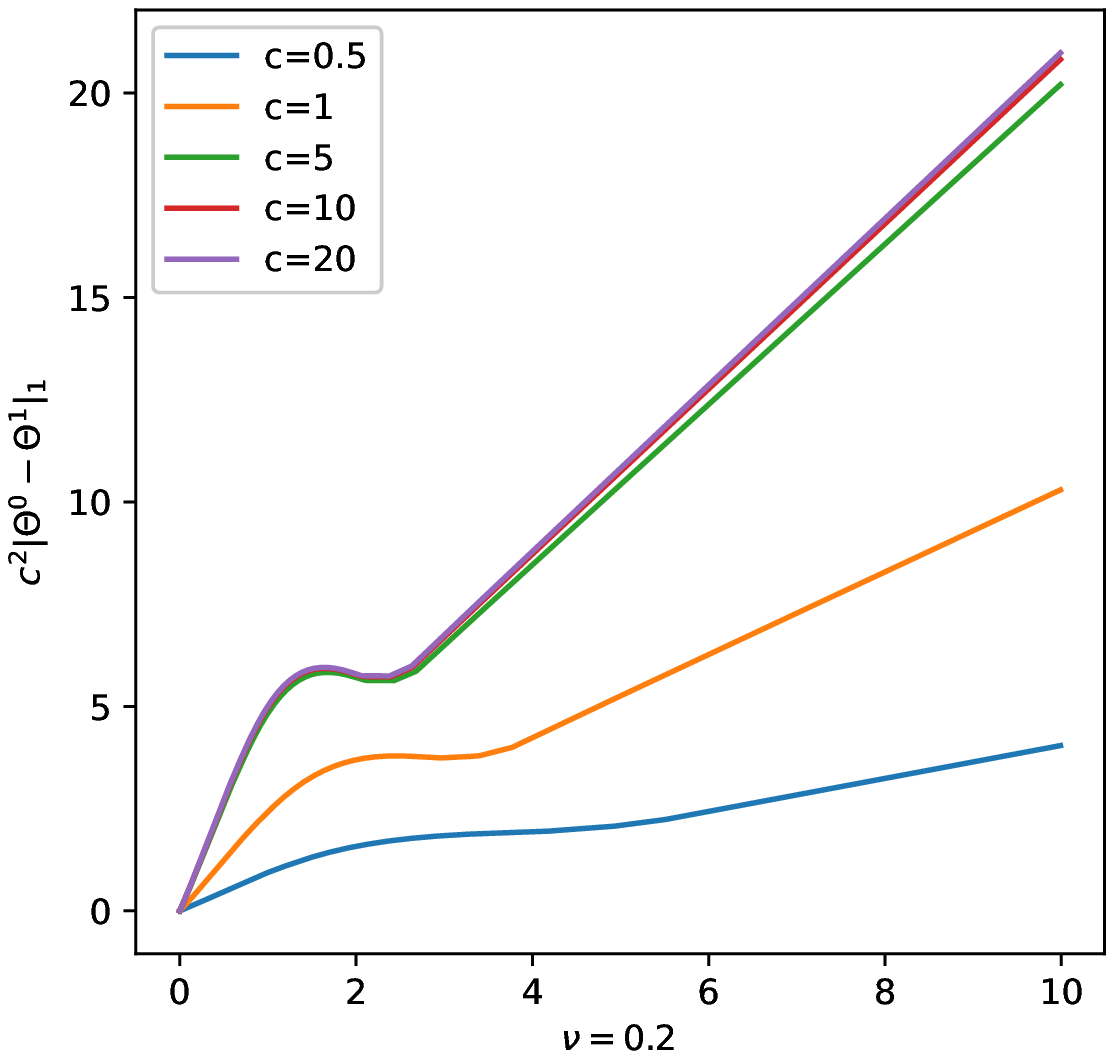}
	\end{subfigure}		
	\begin{subfigure}{0.3\textwidth}
	\centering
	\includegraphics[scale=.4]{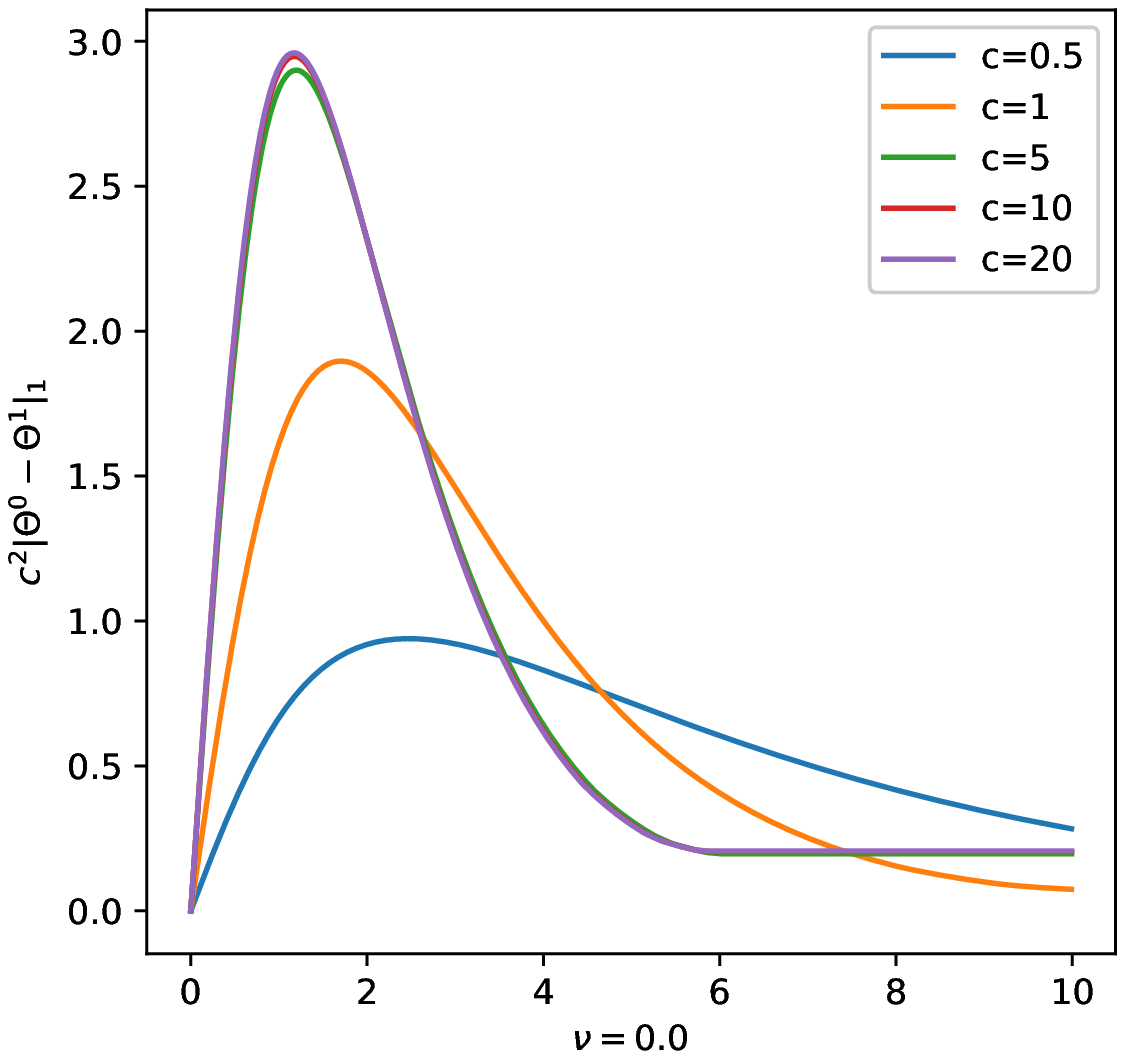}
\end{subfigure}		
	\caption{time-evolution of $ \log \mathcal{D}({\Theta^1} )$ for $ c=0.5,1,5,10 $.}
	\label{Fig3}
\end{figure}
In Figure 3, we study the temporal evolution of $\| \Theta^0 - \Theta^1 \|_1$ over time by changing the speed of light $c = 0.5, 1, 5, 10, 20$ and tries to see whether 
$c^2 \sup_{0 \leq t < T} \| \Theta^0(t) - \Theta^1(t) \|$ is order of ${\mathcal O}(T)$ or not. The first plot of Figure 3 is for the heterogeneous ensemble case, when natural frequencies are non-identical. The second plot is for a homogeneous ensemble case, when all natural frequencies are identical to $ 0.2 $. The last plot is when every natural frequencies are all zero.

The time-evolution of $c^2 \sup_{0 \leq t < \infty} \| \Theta^0(t) - \Theta^1(t) \|$ are almost identical for $ c=5, 10, 20 $. These plots show that the following estimatation holds:
\[  \sup_{0 \leq t \leq T}  \| \Theta^c(t) - \Theta^{\infty}(t) \|_1 \leq    \mathcal{O}(c^{-2}) T.      \]
One can observe that when natural frequencies are heterogeneous and natural frequencies are not equal to zero, the quanitty $c^2 \sup_{0 \leq t < \infty} \| \Theta^0(t) - \Theta^1(t) \|$ is not bounded. The numerical results suggest that $\ell^1 $ distance seems to increase linearly in time.  On the other hand, when all natural frequencies are equal to zero, $c^2 \sup_{0 \leq t < \infty} \| \Theta^0(t) - \Theta^1(t) \|_1$ seems to converge to certain positive value as $ t \to \infty $.
 
Quick observation of $ \sum \dot{\theta} $ suggests this numerical results. For a solution to the Kuramoto model,  $ \sum_{\theta_i \in \Theta^0} \dot{\theta_i}$ is always equal to $ \sum \nu $ where as in general case $  \sum_{\theta_i \in \Theta^1} \dot{\theta_i}$ will not be equal to $ \sum \nu $. Hence the phase-locked state will be slowly drift apart by velocity of $  \dfrac{1}{N} \big( \sum_{\theta_i \in \Theta^1} \dot{\theta_i}-\sum_{\theta_i \in \Theta^0} \dot{\theta_i} \big)$. But when every natural frequency are equal to zero, every agents in $ \Theta^1 $ will converge to certain point, and $ \dfrac{1}{N}\sum_{\theta_i \in \Theta^1} \dot{\theta_i} $ will be equal to zero, which can be verified analytically through Lemma \ref{L3.2}.

%In Figures \ref{Fig2}, \ref{Fig3}, \ref{Fig4} and \ref{Fig5}, we plot the positions of the particles at the time $ t=30 $ for complete phase locking state in homogeneous model, i.e. $ \nu_i=0 \quad i=1,\cdots,N$. Every parameters in this simulations are identical except the speed of light. Figure \ref{Fig2} is the phase-locking state when $ c=1 $ and Figure \ref{Fig4} is when $ c=5 $\ref{Fig3}. \ref{Fig5} is magnified version of \ref{Fig2} and \ref{Fig4}. 
%Note that the red dots $ \Theta_C $ are identical in both figures since they are not effected by the speed of light.
%
%Also Figure\ref{Fig5} shows that $ \Theta_1,\Theta_2,\Theta_3 $ are almost identical to non-relativistic solution $\Theta_C $.
%
%We further develope this simulation and shows how much approximation error can emerge from the lower light speed. We perform simulations for $c=1,2,5$ and compared the results.
%\begin{figure}[h!]
%	\centering
%	\includegraphics[scale=0.6]{Approximation in time}
%	\caption{Graph of $L^2(\Theta_1(t)-\Theta_2(t))$ for $c=1,2,5$.}
%	\label{Fig6}
%\end{figure}
%
%We calculated the $ L2 $-norm of $(\Theta_1(t)-\Theta_2(t))$ for every time $ t $ and for few $ c $. One can observe that the approximation is $ \dot{\Theta} $ is large, but it decreases as each particles converges to phase-locking state.
%The numeric simulation shows that when the speed of light $ c $ becomes higher the approximation error decresea sharply.

\section{Conclusion} \label{sec:6}
In this paper, we have proposed a generalized Kuramoto model for phase synchronization and investigated its emergent behaviors. Our proposed model is quite general enough to incorporate the relativistic Kuramoto model derived from the relativistic Cucker-Smale model on the unit sphere. As a first set of results, we provided several sufficient frameworks for the complete synchronization for homogeneous and inhomogeneous ensembles. Our sufficient frameworks are given in terms of some conditions on the initial data and system parameters such as the coupling strength and natural frequencies. Compared to the Kuramoto model, our frameworks are rather restrictive in the sense that the initial data for complete synchronization is not generic, although the proposed model can be written as a gradient-like system. Thus, it would be interesting whether the complete synchronization holds for generic initial data in a large coupling regime. Second, we studied a non-relativistic limit for the relativistic Kuramoto model by analyzing temporal evolution of $\ell^1$-discrepancy between the non-relativistic Kuramoto model and the relativistic Kuramoto model issued from the same initial configuration. Our rough bound for such an $\ell^1$-discrepancy is the order of $\mbox{size of time-interval } \times c^{-2}$. Thus, in a finite-time interval, we show that the $\ell^1$-discrepancy decays to zero at the order of $c^{-2}$ in a non-relativistic limit $c \to \infty$. Again, it would be very interesting to see whether this non-relativistic limit can be made uniformly in time at least for some admissible set of initial data.  We leave this interesting problem in a future work.

\end{document}